\documentclass[11pt,reqno]{amsart}
\usepackage[utf8]{inputenc}
\usepackage[english]{babel}
\usepackage{ifthen} % if conditioning

\usepackage{amsmath} % math formulas
\usepackage{amssymb} % math symbols
\usepackage{mathrsfs} % matchcal
\usepackage{mathtools} % more math symbols
\usepackage{bbm} % blackboard style cm fonts
\usepackage{cancel} % to cancel the formula
\usepackage{mathdots} % math dots
\usepackage{xfrac} % for nice fractions

\usepackage{amsthm} % theorem enviroments
\usepackage[shortlabels]{enumitem}

\usepackage{lipsum} % lipsum
\usepackage{verbatim} % comments
\usepackage{pifont} % dingbats, symbols
\usepackage[dvipsnames]{xcolor} % colouring
\usepackage{lmodern} % more fonts
\usepackage{marginnote} % za margine

\usepackage{wrapfig} % wrap text around figures
\usepackage{caption} % customize captions

\usepackage{array} % more options for tables
\usepackage{multirow} % merging rows/columns
\usepackage{multicol} % multicols
\usepackage{colortbl} % coloring tables
\usepackage{makecell} % to break line in a cell
\usepackage{arydshln} % add dashed line

\usepackage{tikz} % macro package for creating graphics
\usetikzlibrary{automata,positioning} % for graphs
\usepackage{tkz-euclide} % for geometry figures
\usepackage[customcolors,shade]{hf-tikz} % for coloring

\usepackage{geometry} % page layout
\usepackage{titleps} % headers, footers
\usepackage{footnote} % footnote

\usepackage{hyperref} % cross-referencing
\usepackage{cleveref} % clever cross-referencing

\newcommand{\R}{\mathbb{R}}

\newcommand{\E}{\mathbb{E}}
\newcommand{\EE}{\mathbb{E}}
\newcommand{\1}{\vect{1}}

\newcommand{\s}{\mathbf{s}}

\renewcommand{\AA}{\mathcal{A}}
\newcommand{\BB}{\mathcal{B}}
\newcommand{\CC}{\mathcal{C}}

\newcommand{\RR}{\mathcal{R}}

\newcommand{\eps}{\varepsilon}
\newcommand{\adj}{\top}

\newcommand{\abs}[1]{\left\lvert #1 \right\rvert}
\newcommand{\norm}[1]{\left\lVert#1\right\rVert}

\newcommand{\brap}[1]{\left( #1 \right)}
\newcommand{\bras}[1]{\left[ #1 \right]}
\newcommand{\brac}[1]{\left\{ #1 \right\}}

\DeclareMathOperator{\Cov}{Cov}
\DeclareMathOperator{\Tr}{Tr}
\DeclareMathOperator{\tr}{tr}
\renewcommand{\vec}{\operatorname{vec}}

\newcommand\mynorm[3]{
    \ifthenelse{\equal{#1}{tr}}{\norm{#2}_{\tr,#3}}{
    \ifthenelse{\equal{#1}{Tr}}{\norm{#2}_{\Tr,#3}}{
    \ifthenelse{\equal{#1}{Etr}}{\norm{#2}_{\E\tr,#3}}{
    \ifthenelse{\equal{#1}{ETr}}{\norm{#2}_{\E\Tr,#3}}{
    \ifthenelse{\equal{#1}{tau}}{\norm{#2}_{\tau,#3}}{
    \ifthenelse{\equal{#1}{tautr}}{\norm{#2}_{\tau\otimes\tr,#3}}{
    \ifthenelse{\equal{#1}{tauTr}}{\norm{#2}_{\tau\otimes\Tr,#3}}{
    \norm{#2}_{#1,#3}}}}}}}}
}

\newcommand{\vect}[1]{\boldsymbol{#1}}

\let\emptyset\varnothing

\newcommand{\qlesssim}{\lesssim_q}
\newcommand{\qgtrsim}{\gtrsim_q}

\newtheorem{theorem}{Theorem}[section]
\newtheorem{proposition}[theorem]{Proposition}
\newtheorem{lemma}[theorem]{Lemma}

\newtheorem{remark}[theorem]{Remark}

\newtheorem{example}[theorem]{Example}
\newtheorem{algorithm}[theorem]{Algorithm}

\theoremstyle{definition}
\newtheorem{definition}[theorem]{Definition}

\newcolumntype{L}[1]{>{\raggedright\let\newline\\\arraybackslash\hspace{0pt}}m{#1}}
\newcolumntype{C}[1]{>{\centering\let\newline\\\arraybackslash\hspace{0pt}}m{#1}}
\newcolumntype{R}[1]{>{\raggedleft\let\newline\\\arraybackslash\hspace{0pt}}m{#1}}

\hypersetup{
    colorlinks=true,
    linkcolor=blue,
    filecolor=magenta,
    urlcolor=red,
    citecolor=blue,
}
\newcommand{\myarrow}[5]{
    \draw[latex-latex, very thick, #5] (\leftaxis, #1*\levelwidth+#3) -- (0.5*\leftaxis+0.5*\rightaxis+#4, #1*\levelwidth+#3) -- (0.5*\leftaxis+0.5*\rightaxis+#4, #2*\levelwidth+#3) -- (\rightaxis, #2*\levelwidth+#3);
}

\newcommand{\legendline}[2]{
    \draw[-, very thick, #2] (\legenaxis-0.75, #1*\levelwidth) -- (\legenaxis, #1*\levelwidth);
}

\newcommand{\mynode}[5]{
    \node[draw, circle, minimum size=16pt, inner sep=0pt, #5] at (#1, #2) (#3) {#4};
}

\newcommand{\myedge}[2]{
    \draw[-,thick] (#1) -- (#2);
}

\newcommand{\mygmbound}[6]{
    \node[] at (#1, #2) {\footnotesize
    $\abs{V(#3)} = #4, \abs{{\color{Purple!80}\smin}} = #5, \abs{{\color{Green!80}\wiso}} = #6$
    };
}

\newcommand{\myrectangle}[5]{
    \draw[-] (#1-#2,#3-#4) -- (#1-#2,#3+#4) -- (#1+#2,#3+#4) -- (#1+#2,#3-#4) -- (#1-#2,#3-#4);
    \node[] at (#1,#3+#4+0.3) {#5};
}

\newcommand{\smin}{S_{\mathrm{min}}}
\newcommand{\wiso}{W_{\mathrm{iso}}}

\newcommand{\flatta}[2]{\AA_{\bras{\,#1\,\mid\,#2\,}}}
\newcommand{\interflatta}[3]{\AA_{\bras{\,#1\,\mid\,#2\,\mid\,#3\,}}}
\newcommand{\interflatty}[3]{Y_{\bras{\,#1\,\mid\,#2\,\mid\,#3\,}}}
\newcommand{\flattb}[2]{\BB_{\bras{\,#1\,\mid\,#2\,}}}

\newcommand{\specialcolor}{NavyBlue}
\newcommand{\qwithcolor}{{\color{\specialcolor} q}}
\newcommand{\onewithcolor}{{\color{\specialcolor} 1}}

\newcommand{\myceco}{\cellcolor{Green!8}}

\newcommand{\myR}{{\fontfamily{qcr}\selectfont R}}
\newcommand{\myC}{{\fontfamily{qcr}\selectfont C}}
\newcommand{\myRC}{{\fontfamily{qcr}\tiny\color{gray}\selectfont RC}}

\newcommand{\mypurple}{Purple!30}
\newcommand{\mygreen}{Green!30}

\title{Matrix Chaos Inequalities and Chaos of Combinatorial Type}

\author[Bandeira]{Afonso S.\ Bandeira}
\address{Department of Mathematics, ETH Z\"urich, Switzerland}
\email{bandeira@math.ethz.ch}

\author[Lucca]{Kevin Lucca}
\address{Department of Mathematics, ETH Z\"urich, Switzerland}
\email{kevin.lucca@ifor.math.ethz.ch}

\author[Nizi\'{c}-Nikolac]{Petar Nizi\'{c}-Nikolac}
\address{Department of Mathematics, ETH Z\"urich, Switzerland}
\email{petar.nizic-nikolac@ifor.math.ethz.ch}

\author[Van Handel]{Ramon van Handel}
\address{PACM, Princeton University, Princeton, NJ 08544, USA}
\email{rvan@math.princeton.edu}

\date{\today}

\begin{document}

\begin{abstract}
    Matrix concentration inequalities and their recently discovered sharp
    counterparts provide powerful tools to bound the spectrum of random
    matrices whose entries are linear functions of independent random
    variables. However, in many applications in theoretical computer science
    and in other areas one encounters more general random matrix models,
    called matrix chaoses, whose entries are polynomials of independent random
    variables. Such models have often been studied on a case-by-case basis
    using ad-hoc methods that can yield suboptimal dimensional factors.

    In this paper we provide general matrix concentration inequalities for
    matrix chaoses, which enable the treatment of such models in a systematic
    manner. These inequalities are expressed in terms of flattenings of the
    coefficients of the matrix chaos. We further identify a special family of
    matrix chaoses of combinatorial type for which the flattening parameters
    can be computed mechanically by a simple rule. This allows us to provide a
    unified treatment of and improved bounds for matrix chaoses that arise in
    a variety of applications, including graph matrices, Khatri-Rao matrices,
    and matrices that arise in average case analysis of the sum-of-squares
    hierarchy.
\end{abstract}

\maketitle

\tableofcontents

%%%%%%%%%%%%%%%%%%%%%%%%%%%%%%%%%%%%%%%%
\section{Introduction}\label{sec:intro}
%%%%%%%%%%%%%%%%%%%%%%%%%%%%%%%%%%%%%%%%

Classical random matrix theory is largely concerned with special models, such as matrices with i.i.d.\ entries, whose spectral properties are understood asymptotically with stunning precision. However, random matrices that appear in applications in theoretical computer science and in other fields often fall outside the scope of the classical models; moreover, these applications typically require an understanding of such models in the nonasymptotic regime.

One of the key advances from this perspective has been the development of a large family of \emph{matrix concentration inequalities} that are widely used in applications. These inequalities can be applied to random matrices whose entries are very general \emph{linear} functions of independent random variables. A prototypical example of such a model is any random matrix with centered jointly gaussian entries (with arbitrary covariance), which can always be represented as
\begin{equation*}\label{eq:gaussianmatrix}
    X = \sum_{i\in[m]} g_{i} A_{i}
\end{equation*}
where $g_1,\dots,g_m$ are i.i.d. standard gaussians and $A_i$ are deterministic matrix coefficients.
In this setting, the noncommutative Khintchine (NCK) inequality of Lust-Piquard and Pisier
\cite[\S 9.8]{Pisier2003IntroductionTO} provides explicitly computable upper and lower bounds on the spectral norm $\|X\|$ that differ only by a logarithmic dimensional factor. This inequality has been extended to non-gaussian models that can be expressed as sums of independent random matrices \cite{Tropp2015AnIT}. These results are used in numerous applications, including average case analysis of spectral methods, algorithms in the sum-of-squares hierarchy, and randomized linear algebra.

Matrix concentration inequalities are usually not sharp and often introduce mild but spurious dimensional factors in the analysis. In recent years, new kinds of inequalities have been developed that are applicable to the same models, but eliminate these dimensional factors and give rise to sharp bounds in many applications \cite{bandeira2023matrix,Brailovskaya2022UniversalityAS,BandeiraEtAl24mc2}. This is achieved by introducing additional parameters that quantify the degree to which random matrices behave like idealized models from free probability theory. We will refer to these inequalities as \emph{strong matrix concentration inequalities}, to distinguish them from their classical counterparts.

While matrix concentration inequalities are extremely versatile, there are large classes of models that cannot be readily understood with these tools. One such class, which we call \emph{matrix chaos}, are matrices whose entries are \emph{polynomials} of independent random variables. For example, in the gaussian case, we will consider random matrices $X$ whose entries are homogeneous square-free polynomials 
of independent gaussian variables $g_1,\dots,g_m$:
\begin{equation*}\label{eq:gaussianchaos} 
    X = \sum_{\substack{i_1,\ldots,i_q\in[m] \\ i_1,\ldots,i_q \text{ distinct}}} g_{i_1}\cdots g_{i_q} A_{i_1,\dots,i_q}.
\end{equation*}
Here $q$ is the order of the chaos and $A_{i_1,\dots,i_q}$ are determinstic matrix coefficients. Such models and their non-gaussian counterparts appear in many  applications; a prominent example are the graph matrices of Potechin et al.~\cite{Meka2015SumofsquaresLB,ahn2016graph}.

The aim of this paper is to develop general \emph{matrix chaos inequalities} that enable the treatment of matrix chaos models in a systematic manner, that are easily applicable in concrete situations, and that give rise to sharp bounds in a variety of applications.

\medskip

\textbf{Contributions and prior work.}
It was understood long ago in operator theory that when linear inequalities of NCK type are available, these can be iterated by means of a systematic procedure to obtain chaos inequalities; see \cite{Haagerup1993BoundedLO} and \cite[Remark 9.8.9]{Pisier2003IntroductionTO}. However, the resulting inequalities were not fully spelled out, and their significance to applications does not appear to have been realized.
Consequently, many (special cases of) inequalities of this kind were repeatedly rediscovered in applied mathematics; see, for example, \cite[Theorem~4.3]{Rauhut2009CirculantAT},
\cite[Theorem~6.8]{Ma2016PolynomialTimeTD}, \cite{Minsker2018MomentIF}, 
\cite[\S 4.4]{Deng2020RandomTP}, \cite[Theorem~6.7]{rajendran2023concentration}, 
\cite{Fan2024KroneckerproductRM}, and \cite{tulsiani2024decoupling}.\footnote{Among these references, the closest in spirit to the approach developed here is the recent work \cite{tulsiani2024decoupling}, which appeared on arxiv after the present paper was submitted for publication.} Furthermore, special matrix chaos models, such as graph matrices \cite{Meka2015SumofsquaresLB,ahn2016graph}, have often been investigated using ad-hoc methods without the benefit of generally applicable tools.

In this paper, we revisit the original operator-theoretic approach for deriving matrix chaos inequalities from their linear counterparts. Besides drawing attention to this simple and natural method, it enables us to achieve a significantly improved toolbox for the study of matrix chaoses that arise in applications. The main contributions of this paper are twofold:

\medskip

\noindent (i)
We will show in section \ref{sec:iteratingMi} that the operator-theoretic approach can be adapted to apply not only to NCK-type inequalities, but also to the recent theory of strong matrix concentration inequalities. This gives rise to \emph{strong matrix chaos inequalities} that yield bounds without  spurious dimensional factors in various settings. By using the linear inequalities as a black box, these inequalities leverage sophisticated tools of random matrix theory and free probability to obtain general bounds that would be difficult to achieve using ad-hoc methods.

\medskip

\noindent(ii)
The basic construction that underpins the operator-theoretic approach will naturally lead us in section~\ref{sec:3} to the consideration of a special class of models that we call \emph{chaos of combinatorial type}, for which all the parameters that appear in our bounds can be computed explicitly by a simple rule. Many matrix chaoses that we have encountered in theoretical computer science applications turn out to be special cases of this class. When that is the case, our methods reduce the study of such models to a nearly trivial computation that often yields improved bounds.

\medskip

It is worth emphasizing that our inequalities provide both upper and lower bounds on the spectral norm of matrix chaoses, which suffice to show in most cases that our bounds are optimal either up to a universal constant or a logarithmic dimensional factor.

In section \ref{sec:applications}, we will illustrate our main results in the context of two notable examples of chaos of combinatorial type: \emph{graph matrices}~\cite{Medarametla2016BoundsOT,ahn2016graph}, which are ubiquitous in the average case analysis of sum-of-squares algorithms~\cite{Meka2015SumofsquaresLB,Barak2016ANT,Potechin2020MachineryFP}; and
\emph{Khatri-Rao matrices}~\cite{KhatriRao1968SoTS}, which have been used in the context of differential privacy~\cite{Kasiviswanathan2010ThePO,De2011LowerBI}.
Beside providing
a unified analysis of the spectrum of these matrices, our techniques yield, in several applications, inequalities without spurious dimensional factors. We will illustrate the latter in the context of the ellipsoid fitting problem (recovering, with a simplified argument, the sharper analysis of graph matrices in~\cite{Hsieh2023EllipsoidFU}); and in a matrix chaos arising in the analysis of a sum-of-squares algorithm for tensor PCA (resulting in a correct-up-to-constants algorithmic guarantee).

In this paper, we focus for simplicity on bounding the spectral norm of homogeneous and square-free matrix chaoses. Our results can be extended to treat more general models, as well as more general spectral statistics such as the smallest singular value.
We defer such extensions and other applications of our techniques to a longer companion manuscript~\cite{IteratedNCK-Journal}.

\medskip

\textbf{Notation.} The following notations will be used throughout this paper.

We write $x \qlesssim y$ if $x \leq C_q y$ for a universal constant $C_q$ that depends only on $q$. When $x \qlesssim y$ and $y \qlesssim x$, we write $x\asymp_q y$. We use $x \lesssim y$ and $x \asymp y$ when the constants are universal.
We denote by $a:b = \{a,a+1,\ldots,b\}$, by $[n]=1:n$, and by $|I|$ the cardinality
of a finite set $I$. We define $x \lor y \coloneqq \max \{x, y\}$.

We will always work with real matrices for simplicity (the results of this paper extend readily to complex matrices). The entries of a matrix $M$ will be denoted $M[i,j]$ or $M_{i,j}$, its adjoint is denoted $M^\adj$, and its operator norm is denoted $\|M\|$.
For a scalar random variable $h$, we denote by $\|h\|_{\psi_2}$ its subgaussian constant and by $\|h\|_{L^p}=(\E|h|^p)^{1/p}$ its $L^p$-norm.
%%%%%%%%%%%%%%%%%%%%%%%%%%%%%%%%%%%%%%%%
\section{Matrix chaos inequalities}\label{sec:iteratingMi}
%%%%%%%%%%%%%%%%%%%%%%%%%%%%%%%%%%%%%%%%

The aim of this section is to formulate the main inequalities of this paper.
In section \ref{subsec:decoupling}, we first introduce the general matrix chaos model and its decoupled version. In section \ref{sec:flattenings}, we introduce the basic notation for tensor flattenings that will be used throughout this paper. The main inequalities are stated in section \ref{sec:mainineq}. Finally, we will outline the basic approach to the proofs in section \ref{subsec:iterapproach}. Most of the proof details will be deferred to Appendix \ref{section:proofs}.

%%%%%%%%%%%%%%%%%%%%%%%%%%%%%%%%%%%%%%%%
\subsection{Matrix chaos and decoupling}\label{subsec:decoupling}
%%%%%%%%%%%%%%%%%%%%%%%%%%%%%%%%%%%%%%%%

The basic model of this paper is a \emph{matrix chaos}
\begin{equation}\label{eq:chaos}
    X = \sum_{\substack{i_1,\ldots,i_q\in[m] \\ i_1,\ldots,i_q \text{ distinct}}} h_{i_1}\cdots h_{i_q} A_{i_1,\dots,i_q}
\end{equation}
of order $q$.
Here $h_1,\ldots,h_m$ are i.i.d.\ copies of a random variable $h$ with zero mean, and
$A_{i_1,\ldots,i_q}$ are deterministic $d_1\times d_2$ matrix coefficients (we will write $d=d_1\vee d_2$). 

We will often consider a decoupled variant of the above model. To this end, let $\vect{h}^{(1)}, \ldots, \vect{h}^{(q)}$ denote i.i.d.\ copies of $\vect{h} \coloneqq \brap{h_1, \ldots, h_m}$. We define the \emph{decoupled matrix chaos} as
\begin{equation}\label{eq:decnongaussianchaos}
    Y = \sum_{i_1,\ldots,i_q\in[m]} h_{i_1}^{(1)}\cdots h_{i_q}^{(q)} A_{i_1,\dots,i_q}.
\end{equation}
Note that in the decoupled case, the coordinates $i_1,\ldots,i_q$ need not be distinct.

The connection between coupled and decoupled chaoses is captured by classical
decoupling inequalities. In the present setting, \cite[Theorem 3.1.1]{Gin1998DecouplingFD} yields the following.

\begin{theorem}[Decoupling inequalities]\label{thm:decoupling}
    Let $X$ be any matrix chaos as in~\eqref{eq:chaos}, and let $Y$ be the
    decoupled matrix chaos as in~\eqref{eq:decnongaussianchaos} defined by the same random variables $h_1,\ldots,h_m$ and matrix coefficients $A_{i_1,\ldots,i_q}$ (where we set $A_{i_1,\ldots,i_q}=0$ when $i_1,\ldots,i_q$ are not distinct). Then we have
    \begin{equation*}
        \E\norm{X} \qlesssim \E\norm{Y}.
    \end{equation*}
    Moreover, this inequality can be reversed
    \begin{equation*}
        \E\norm{Y} \qlesssim \E\norm{X}
    \end{equation*}
    provided that the matrix coefficients are assumed to be symmetric in the sense that
    $A_{i_1, \ldots, i_q} = A_{i_{\pi(1)}, \ldots, i_{\pi(q)}}$ for every permutation $\pi$
    of $[q]$.
\end{theorem}

The iteration argument that forms the basis for our proofs will rely crucially on the independence structure of the decoupled model. 
As decoupled chaoses arise in applications in their own right, we will formulate our main inequalities for decoupled chaoses \eqref{eq:decnongaussianchaos} and take for granted in the sequel that these inequalities can also be applied for coupled chaoses \eqref{eq:chaos} by virtue of Theorem \ref{thm:decoupling}.

\begin{remark}[Lower bounding $\E\norm{X}$]
    The lower bound in Theorem~\ref{thm:decoupling} has an additional assumption that the matrix
    coefficients are symmetric. This assumption is necessary: consider, for example, the chaos $Y = g_1^{(1)}g_2^{(2)} - g_2^{(1)}g_1^{(2)}$ whose coupled version vanishes $X = 0$. On the other hand, as the matrix coefficients in \eqref{eq:chaos} may clearly be chosen to be symmetric without loss of generality, this does not present any fundamental restriction to obtaining
    lower bounds on $\E\norm{X}$.
\end{remark}

\begin{remark}[More general chaoses]\label{rmk:non_diag_free_chaos}
    In this paper, we work only with homogeneous square-free chaoses \eqref{eq:chaos}.
    However, non-homogeneous or non-square-free matrix chaoses can often be treated using
    similar methods. Such models are addressed in the longer companion manuscript~\cite{IteratedNCK-Journal}.
\end{remark}

%%%%%%%%%%%%%%%%%%%%%%%%%%%%%%%%%%%%%%%%
\subsection{Flattenings}
\label{sec:flattenings}
%%%%%%%%%%%%%%%%%%%%%%%%%%%%%%%%%%%%%%%%

Fix a decoupled matrix chaos as in \eqref{eq:decnongaussianchaos}. It will be convenient to view the matrix coefficients $A_{i_1,\ldots,i_q}$ of the chaos as defining a tensor $\AA$ of order $q+2$ by
$$
    \AA_{i_1, \ldots, i_q, i_{q+1}, i_{q+2}} \coloneqq \brap{A_{i_1, \ldots, i_q}}_{i_{q+1}, i_{q+2}}.
$$
Here the first $q$ coordinates (which we call \emph{chaos coordinates}) range from $1$ to $m$ and the last two (which we call \emph{matrix coordinates}) range from $1$ to $d_1$ and $1$ to $d_2$, respectively.

The main inequalities of this paper will be defined in terms of the norms of \emph{flattenings} of the tensor $\AA$ that are defined as follows. Denote by $e_i$ the $i$th element of the standard coordinate basis, viewed as a column vector. Then for any subsets $R,C\subseteq[q+2]$, we define
the matrix
\begin{equation}\label{eq:flattenings}
    \flatta{R}{C} \coloneqq \sum_{\substack{i_1, \ldots, i_q \in [m]\\i_{q+1} \in [d_1], i_{q+2} \in [d_2]}} \brap{\bigotimes_{t \in R} e_{i_t}} \otimes \brap{\bigotimes_{t \in C} e_{i_t}^\top} \AA_{i_1, \ldots, i_{q+2}}.
\end{equation}
This definition is easiest to interpret when $R=[q+2]\backslash C$: in this case, $\flatta{R}{C}$ is the matrix whose rows are indexed by the coordinates in the row set $R$, whose columns are indexed by the coordinates in the column set $C$, and whose entries are the corresponding entries of $\AA$. For example, if $q=2$ and $R=\{1,3\}$, $C=\{2,4\}$, then the associated flattening $\flatta{R}{C}$ is the $md_1\times md_2$ matrix with entries
$(\flatta{R}{C})_{(i_1,i_3),(i_2,i_4)} = \AA_{i_1,i_2,i_3,i_4}$. However, we will also encounter flattenings where the same coordinate may appear simultaneously in $R$ and $C$, which corresponds to diagonalization. For example, if $q=1$ and $R=\{1,2\}$, $C=\{1,3\}$, then $(\flatta{R}{C})_{(i_1,i_2),(i_1',i_3)} = 1_{i_1=i_1'}\AA_{i_1,i_2,i_3}$.

%%%%%%%%%%%%%%%%%%%%%%%%%%%%%%%%%%%%%%%%
\subsection{The main inequalities}
\label{sec:mainineq}
%%%%%%%%%%%%%%%%%%%%%%%%%%%%%%%%%%%%%%%%

We now formulate the main inequalities of this paper, which bound the spectral norm of a matrix chaos in terms of the spectral norms of flattenings of the coefficient tensor $\AA$.
All these inequalities will be derived by an iteration argument from an underlying matrix concentration inequality for linear random matrices. The basic idea behind the iteration method will be explained in section \ref{subsec:iterapproach}. We postpone the detailed proofs to Appendix~\ref{section:proofs}. 

Our main inequalities will be stated for decoupled chaoses as in \eqref{eq:decnongaussianchaos}. Their extension to coupled chaoses as in \eqref{eq:chaos}
is immediate by Theorem~\ref{thm:decoupling}. We focus for simplicity on expectation bounds; tail bounds can then be deduced using concentration tools (e.g., \cite{adamczaketal2021} or as in \cite[Lemma 7.6]{pisier2013subexponential}).

\subsubsection{Iterated NCK}

The simplest matrix concentration inequality for linear random matrices is the noncommutative Khintchine inequality (NCK) \cite[\S 9.8]{Pisier2003IntroductionTO}, see Theorem \ref{thm:nck} in the appendix. We begin by formulating the corresponding matrix chaos inequality.

To this end, we now introduce the basic parameter that controls the leading order behavior of matrix chaoses. A flattening is said to be a $\boldsymbol{\sigma}$\textbf{-flattening} if
$R=[q+2]\backslash C$ and if the original matrix coordinates are kept as row and column coordinates, that is, $q+1\in R$ and $q+2\in C$. In this case, $\flatta{R}{C}$ is an $m^{|R|-1}d_1\times m^{|C|-1}d_2$ matrix. We now define 
\begin{equation}\label{eq:sigmadefinition}
    \sigma(\AA) \coloneqq \max_{\substack{R=[q+2]\backslash C\\q+1\in R, q+2\in C}} \norm{\flatta{R}{C}},
\end{equation}
that is, $\sigma(\AA)$ is the largest spectral norm of all $\sigma$-flattenings. We can now formulate the iterated NCK inequality, whose proof will be given in Appendix~\ref{sec:proofIterNCK}.

\begin{theorem}[Iterated NCK]\label{thm:iteratednck}
    Let $Y$ be a decoupled chaos as in \eqref{eq:decnongaussianchaos}. Then
    \begin{equation*}
        \|h\|_{L_1}^q\sigma(\AA) \qlesssim \E\norm{Y} \qlesssim 
        \|h\|_{\psi_2}^q\log(d+m)^{\frac{q}{2}} \sigma(\AA).
    \end{equation*}
    Alternatively, the upper bound remains valid if $\|h\|_{\psi_2}$ is replaced by $\|h\|_{L^{\log m}}$.
\end{theorem}

Note, for example, that $\|h\|_{L^1},\|h\|_{\psi_2}\asymp 1$ if $h$ is a standard Gaussian or Rademacher variable. When this is the case, Theorem \ref{thm:iteratednck} states that the parameter $\sigma(\AA)$ captures the spectral norm of any matrix chaos up to a logarithmic dimensional factor.

\subsubsection{Iterated strong NCK}

The drawback of Theorem \ref{thm:iteratednck} is that the dimensional factor in the upper bound often proves to be suboptimal. For subgaussian random matrices, sharp bounds can often be achieved by using instead the strong NCK inequality of \cite{bandeira2023matrix}, see Theorem~\ref{thm:freenck} in the appendix. We now formulate a corresponding matrix chaos inequality.

A flattening is said to be a $\boldsymbol{v}$\textbf{-flattening} if $R=[q+2]\backslash C$ is nonempty and if the original matrix coordinates are both assigned to be column coordinates, that is, $q+1\in C$ and $q+2\in C$. Define
\begin{equation}\label{eq:vdefinition}
    v(\AA) \coloneqq \max_{\substack{R= [q+2]\backslash C\\q+1, q+2\in C\\R \neq \emptyset}} \norm{\flatta{R}{C}},
\end{equation}
that is, $v(\AA)$ is the largest spectral norm of all $v$-flattenings.
We can now formulate the iterated strong NCK inequality, whose proof will be given in
Appendix~\ref{sec:proofIterStrongNCK}.

\begin{theorem}[Iterated strong NCK]\label{thm:iteratedfreenck}
    Let $Y$ be a decoupled chaos as in \eqref{eq:decnongaussianchaos}. Then
    \begin{equation*}
        \E\norm{Y} \qlesssim 
        \|h\|_{\psi_2}^q \brap{\sigma(\AA) + \log(d+m)^{\frac{q+2}{2}} v(\AA)}.
    \end{equation*}
\end{theorem}

\smallskip

The corresponding lower bound on $\E\norm{Y}$ follows already from Theorem \ref{thm:iteratednck}. The significance of Theorem \ref{thm:iteratedfreenck} is that when $v(\AA)\ll\sigma(\AA)$, the logarithmic factor in Theorem \ref{thm:iteratednck} is eliminated.

\subsubsection{Iterated matrix Rosenthal}

The above results yield matching upper and lower bounds for matrix chaoses that are based on regularly behaved random variables $h_1,...,h_m$, such as Gaussians or Rademachers. However, they may result in poor bounds in situations where $\|h\|_{\psi_2}$ is very large or $\|h\|_{L^1}$ is very small. A typical situation of this kind that arises frequently in practice is in the study of sparse models, where $h$ is a standardized (i.e., normalized to have zero mean and unit variance) $\mathrm{Bern}(p)$ random variable. In this case, it is readily verified that $\|h\|_{\psi_2}\to\infty$ and $\|h\|_{L^1}\to 0$ in the sparse regime $p\to 0$, which causes the previous bounds to diverge.

For linear random matrices, this issue can be surmounted by using inequalities of Rosenthal type \cite{junge2013,Mackey_2014,Brailovskaya2022UniversalityAS},
see Theorem \ref{thm:rosenthal} in the appendix. We now formulate a corresponding matrix chaos inequality. The strong form will be given in the next section.

A flattening is said to be an $\boldsymbol{r}$\textbf{-flattening} if the original matrix coordinates are kept as row and column coordinates, that is, $q+1\in R$ and $q+2\in C$, but there is at least one of the $q$ chaos coordinates that appears both in $R$ and $C$. We now define
\begin{equation}\label{eq:rdefinition}
    r(\AA) \coloneqq \max_{\substack{R \cup C = [q+2]\\q+1\in R, q+2\in C\\\emptyset \neq R \cap C \subseteq [q]}} \norm{\flatta{R}{C}},
\end{equation}
that is, $r(\AA)$ as the largest spectral norm of all $r$-flattenings. We can now formulate an iterated matrix Rosenthal inequality, whose proof is given in Appendix~\ref{sec:proofIteratedRosenthal}.

\begin{theorem}[Iterated matrix Rosenthal]\label{thm:iteratedrose}
    Let $Y$ be a decoupled chaos as in~\eqref{eq:decnongaussianchaos}. Assume that $h$ has unit variance, and define the 
    parameter $\alpha(h) = \|h\|_{L^{\log(d+m)}}$. Then we have
    \begin{equation*}
    \sigma(\AA) - C_q\alpha(h)^q \log(d+m)^{\frac{q}{2}}r(\AA)
        \qlesssim
        \E\norm{Y} \qlesssim \log(d+m)^{\frac{q}{2}} \sigma(\AA)+ \alpha(h)^q \log(d+m)^{\frac{q+1}{2}}r(\AA),
    \end{equation*}
    where $C_q$ is a constant that depends only on $q$.
\end{theorem}

This result may be viewed as an analogue of the iterated NCK inequality 
(Theorem \ref{thm:iteratednck})
where the distributional parameter $\alpha(h)$ only appears in the second-order term that is controlled by $r(\AA)$. Therefore, when $r(\mathcal{A})\ll\sigma(\mathcal{A})$, the parameter $\sigma(\mathcal{A})$ captures the spectral norm up to a logarithmic factor even in (e.g., sparse) situations where $\alpha(h)$ may diverge.

\subsubsection{Iterated strong matrix Rosenthal}

Just as the strong NCK inequality eliminates the dimensional factor
in the NCK inequality in many situations, there is an analogous strong form of the matrix Rosenthal inequality  \cite{Brailovskaya2022UniversalityAS}, see Theorem \ref{thm:linearuniversalityrosenthal}. We now formulate a corresponding matrix chaos inequality, whose proof will be given in Appendix~\ref{sec:proof:thm:iteratedstrongrose}.

\begin{theorem}[Iterated strong Matrix Rosenthal]\label{thm:iteratedstrongrose}
    Let $Y$ be a decoupled chaos as in~\eqref{eq:decnongaussianchaos}. Assume that $h$ has unit variance, and define the 
    parameter $\alpha(h) = \|h\|_{L^{\log(d+m)}}$. Then 
    \begin{equation*}
        \E\norm{Y} \qlesssim \sigma(\AA)+ \alpha(h)^q \log(d+m)^{\frac{q+3}{2}} v(\AA).
    \end{equation*}
\end{theorem}

\smallskip

Let us note that the inequality $r(\AA) \leq v(\AA)$ always holds (Lemma~\ref{lem:comparison_v_r}), so that $r(\AA)$ need not be computed when applying an inequality in which $v(\AA)$ already appears. In particular, the lower bound corresponding to Theorem~\ref{thm:iteratedstrongrose} follows from Theorem~\ref{thm:iteratedrose}. 
The significance of Theorem~\ref{thm:iteratedstrongrose} is that when $\alpha(h)^qv(\AA)\ll\sigma(\AA)$, we obtain $\E\norm{Y}\asymp_q \sigma(\AA)$ without
a logarithmic factor.

%%%%%%%%%%%%%%%%%%%%%%%%%%%%%%%%%%%%%%%%
\subsection{Iteration approach}\label{subsec:iterapproach}
%%%%%%%%%%%%%%%%%%%%%%%%%%%%%%%%%%%%%%%%

We now outline the basic iteration approach to the proofs of our matrix chaos inequalities. For NCK this approach dates back at least to \cite{Haagerup1993BoundedLO}, and we will explain how it can be adapted to capture the strong inequalities. We defer detailed proofs to Appendix~\ref{section:proofs}.

\subsubsection{The linear case}\label{subsubsec:flattlincase}

The key observation behind the proofs is that the linear matrix concentration inequalities can be reinterpreted in terms of flattenings. Once they have been reformulated in this manner, the chaos inequalities will follow seamlessly by induction.

Let us begin by illustrating the case of NCK. Consider a matrix chaos $Y$ as in \eqref{eq:decnongaussianchaos} of order $q=1$, that is, $Y=\sum_{i\in[m]} h_i A_i$. The NCK inequality (Theorem \ref{thm:nck}) states that
\begin{equation}
\label{eq:nckintro}
    \E \|Y\| \lesssim \|h\|_{\psi_2}\log(d)^{\frac{1}{2}} (\sigma_R(Y)\vee\sigma_C(Y))
\end{equation}
with
$$
    \sigma_R(Y) \coloneqq \Bigg\|\sum_{i\in[m]} A_i^\top A_i\Bigg\|^{\frac{1}{2}},\qquad\qquad\qquad\qquad
    \sigma_C(Y) \coloneqq \Bigg\|\sum_{i\in[m]} A_i A_i^\top\Bigg\|^{\frac{1}{2}}.
$$
To reformulate this inequality in terms of flattenings, note that
\begin{equation*}
    \quad\flatta{\onewithcolor,2}{3} = 
    \sum_{i\in[m]} e_i\otimes A_i =
    \begin{pmatrix} A_1 \\ \vdots \\ A_m\end{pmatrix}, \qquad\quad
\flatta{2}{\onewithcolor,3} = 
    \sum_{i\in[m]} e_i^\top\otimes A_i =
    \begin{pmatrix} A_1 & \cdots & A_m\end{pmatrix}.
\end{equation*}
As
$$
    \flatta{\onewithcolor,2}{3}^\adj \flatta{\onewithcolor,2}{3} = \sum_{i\in[m]} A_i^\adj A_i,\qquad\qquad\qquad\quad
    \flatta{2}{\onewithcolor,3} \flatta{2}{\onewithcolor,3}^\adj = \sum_{i\in[m]} A_i A_i^\adj,
$$
we have clearly shown that $\sigma_R(Y)=\|\flatta{\onewithcolor,2}{3}\|$ and
$\sigma_C(Y)=\|\flatta{2}{\onewithcolor,3}\|$. Note that in this notation, the NCK inequality \eqref{eq:nckintro} is essentially recovered as the $q=1$ case of Theorem \ref{thm:iteratednck}.

The strong NCK, Rosenthal, and strong Rosenthal inequalities (Theorems \ref{thm:freenck}, \ref{thm:rosenthal}, and \ref{thm:linearuniversalityrosenthal}) involve two additional matrix parameters 
$$
    v(Y) \coloneqq \|\Cov(Y)\,\|^{\frac{1}{2}},\qquad\qquad
    r(Y) \coloneqq \max_{i\in[m]}\|A_i\|,
$$
where $\Cov(Y)$ denotes the covariance matrix of the entries of $Y$. We now observe that these parameters can also be reformulated in terms of flattenings. To this end, note that
\begin{alignat*}{3}
&   \flatta{\onewithcolor}{2,3} &&= 
    \sum_{i\in[m]} e_i\otimes \vec(A_i)^\adj &&=
    \begin{pmatrix}\vec(A_1)^\adj \\ \vdots \\ \vec(A_m)^\adj\end{pmatrix}, \\
&   \flatta{\onewithcolor,2}{\onewithcolor,3} &&=
    \sum_{i\in[m]} e_i\otimes e_i^\adj \otimes A_i &&=
    \begin{pmatrix}A_1 & \cdots & 0 \\ \vdots & \ddots & \vdots \\ 0 & \cdots & A_m\end{pmatrix},
\end{alignat*}
where $\vec(\cdot)$ denotes the operation that arranges all the entries of a matrix in a column vector. As $\flatta{\onewithcolor}{2,3}^\adj\flatta{\onewithcolor}{2,3} = \Cov(Y)$, it follows directly that
$v(Y) = \|\flatta{\onewithcolor}{2,3}\|$. The operator norm of a block-diagonal matrix equals the maximum operator norm of its blocks, hence $r(Y)=\|\flatta{\onewithcolor,2}{\onewithcolor,3}\|$.
In this notation, the strong NCK and (strong) Rosenthal inequalities are again essentially recovered as the $q=1$ case of the corresponding matrix chaos inequalities.

\subsubsection{Iteration}
\label{sec:iteration}

Now let $Y$ be a decoupled chaos as in~\eqref{eq:decnongaussianchaos} of order $q\ge 2$. If we condition on the random variables associated with the first $q-1$ chaos coordinates (the vectors $\vect{h}^{(1)}, \ldots, \vect{h}^{(q-1)}$), then $Y$ can be written as a linear chaos with random coefficients:
\begin{equation}\label{eq:YaslinearChaos}
    Y = \sum_{i_{\qwithcolor}} h_{i_\qwithcolor}^{(\qwithcolor)} \brap{\sum_{i_1, \ldots, i_{q-1}} h_{i_1}^{(1)} \cdots h_{i_{q-1}}^{(q-1)} A_{i_1, \ldots, i_{q}}} = \sum_{i_\qwithcolor} h_{i_\qwithcolor}^{(\qwithcolor)} B_{i_\qwithcolor}.
\end{equation}
Applying the linear inequalities to the random matrix $Y = \sum_{i_\qwithcolor} h_{i_\qwithcolor}^{(\qwithcolor)} B_{i_\qwithcolor}$ yields upper bounds in terms of four possible flattenings, each of which can itself be interpreted as a matrix chaos of order $q-1$.
For example, the $\sigma_R(Y)$ parameter of this random matrix is the norm of
\begin{equation}\label{eq:distributivityintermediate}
    \begin{aligned}
    \sum_{i_{\qwithcolor}} e_{i_\qwithcolor}\otimes B_{i_\qwithcolor}
    &= \sum_{i_{\qwithcolor}} e_{i_\qwithcolor} \otimes \brap{\sum_{i_1, \ldots, i_{q-1}} h_{i_1}^{(1)} \cdots h_{i_{q-1}}^{(q-1)} A_{i_1, \ldots, i_{q}}}\\
    &= \sum_{i_{\qwithcolor}} e_{i_\qwithcolor} \otimes \brap{\sum_{i_1, \ldots, i_{q-1}} h_{i_1}^{(1)} \cdots h_{i_{q-1}}^{(q-1)} \brap{\sum_{i_{q+1}, i_{q+2}} e_{i_{q+1}} \otimes e_{i_{q+2}}^\adj \, \AA_{i_1, \ldots, i_{q+2}}}}\\
    &= \sum_{i_1, \ldots, i_{q-1}} h_{i_1}^{(1)} \cdots h_{i_{q-1}}^{(q-1)} 
    \brap{
    \sum_{i_\qwithcolor,i_{q+1},i_{q+2}}
    \brap{e_{i_\qwithcolor} \otimes e_{i_{q+1}}}\otimes e_{i_{q+2}}^\adj\, \AA_{i_1, \ldots, i_{q+2}} },
    \end{aligned}
\end{equation}
which is a decoupled matrix chaos of order $q-1$ with matrix coefficients of dimension 
$md_1 \times d_2$.
Analogous expressions hold for the remaining matrix parameters. In this manner, the expected norm of a matrix chaos of order $q$ is bounded by the expected norms of matrix chaoses of order $q-1$, and the proofs can proceed by induction on $q$.

To formalize the above procedure, we introduce the following notation.
Given $Z, R, C \subseteq [q+2]$ with $Z \subseteq [q]$, we define the
\textbf{intermediate flattening}
\begin{equation}\label{eq:intermediateflattenings}
    \interflatty{Z}{R}{C} \coloneqq \sum_{i_1, \ldots, i_{q+2}} 
    \brap{\bigotimes_{t \in R} e_{i_t}} \otimes\brap{\bigotimes_{t \in C} e_{i_t}^\adj} 
    \brap{\prod_{t \in Z} h_{i_t}^{(t)}} 
    \AA_{i_1, \ldots, i_{q+2}},
\end{equation}
Note that $\interflatty{Z}{R}{C}$ (with $Z\neq\emptyset$) is a decoupled matrix chaos of order $|Z|$. We will denote by $\interflatta{Z}{R}{C}$ the tensor of order $|Z|+2$ associated with the chaos $\interflatty{Z}{R}{C}$. The intermediate flattening in~\eqref{eq:distributivityintermediate} corresponds precisely to $\interflatty{1:q-1}{\qwithcolor, q+1}{q+2}$.

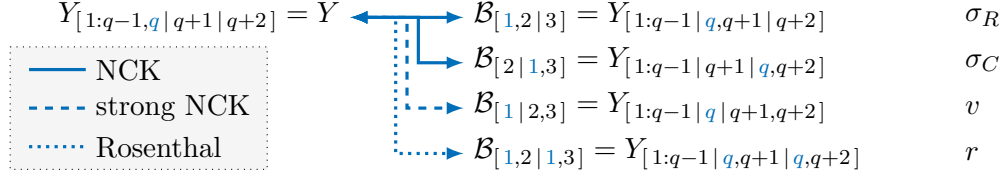
\begin{figure}
    \begin{tikzpicture}
        \def\textaxis{-8.75}; \def\leftmost{-10.5}; \def\rightmost{7};
        \def\levelwidth{0.6};
        
        \def\leftaxis{0};
        \def\rightaxis{1.5};
        \def\legenaxis{-3.5};
        \def\offset{0.15}
    
        \node[left] at (\leftaxis, 1.5*\levelwidth) {$\interflatty{1:q-1, \qwithcolor}{q+1}{q+2} = Y$};
        \node[right] at (\rightaxis, 1.5*\levelwidth) {$\flattb{\onewithcolor, 2}{3} = \interflatty{1:q-1}{\qwithcolor, q+1}{q+2}$};
        \node[right] at (\rightaxis, 0.5*\levelwidth) {$\flattb{2}{\onewithcolor, 3} = \interflatty{1:q-1}{q+1}{\qwithcolor, q+2}$};
        \node[right] at (\rightaxis, -0.5*\levelwidth) {$\flattb{\onewithcolor}{2, 3} = \interflatty{1:q-1}{\qwithcolor }{q+1, q+2}$};
        \node[right] at (\rightaxis, -1.5*\levelwidth) {$\flattb{\onewithcolor, 2}{\onewithcolor, 3} = \interflatty{1:q-1}{\qwithcolor, q+1}{\qwithcolor, q+2}$};
        
        \node[right] at (\rightaxis+6.5, 1.5*\levelwidth) {$\sigma_R$};
        \node[right] at (\rightaxis+6.5, 0.5*\levelwidth) {$\sigma_C$};
        \node[right] at (\rightaxis+6.5, -0.5*\levelwidth) {$v$};
        \node[right] at (\rightaxis+6.5, -1.5*\levelwidth) {$r$};
    
        \myarrow{1.5}{-0.5}{0}{0}{\specialcolor,dashed}
        \myarrow{1.5}{-1.5}{0}{-\offset}{\specialcolor,dotted}
        \myarrow{1.5}{1.5}{0}{\offset}{\specialcolor}
        \myarrow{1.5}{0.5}{0}{\offset}{\specialcolor}
    
        \draw[-, dotted, fill=gray!8] (\legenaxis-1, 0.85*\levelwidth) -- (\legenaxis+2.4, 0.85*\levelwidth) -- (\legenaxis+2.4, -1.85*\levelwidth) --  (\legenaxis-1, -1.85*\levelwidth) -- (\legenaxis-1, 0.85*\levelwidth);
    
        \legendline{0.35}{\specialcolor}
        \node[right] at (\legenaxis, 0.35*\levelwidth) {NCK};
    
        \legendline{-0.5}{\specialcolor,dashed}
        \node[right] at (\legenaxis, -0.5*\levelwidth) {strong NCK};
    
        \legendline{-1.35}{\specialcolor,dotted}
        \node[right] at (\legenaxis, -1.35*\levelwidth) {Rosenthal};
    \end{tikzpicture}
    \caption{Intermediate flattenings that arise from each matrix parameter. Here $\BB$ is the (random) tensor of order $3$ associated to the linear chaos $\sum_{i_\onewithcolor} h_{i_\onewithcolor} B_{i_\onewithcolor}$ in~\eqref{eq:YaslinearChaos}.}
    \label{fig:branching}
\end{figure}

Using this notation, applying the linear matrix concentration inequalities to the original chaos 
$Y = \interflatty{1:q-1, \qwithcolor}{q+1}{q+2}$ of order $q$ yields bounds in terms of four intermediate chaoses of order $q-1$ as described in Figure~\ref{fig:branching}. To prove the our main inequalities, we can iterate this procedure until the order of the (intermediate) flattenings has been reduced to zero. The resulting final flattenings $\interflatty{\emptyset}{R}{C}$ are deterministic and are equal to both $\interflatta{\emptyset}{R}{C}$ and $\flatta{R}{C}$.
In practice, this procedure is easily implemented by induction on $q$.
The details are deferred to Appendix~\ref{section:proofs}.

%%%%%%%%%%%%%%%%%%%%%%%%%%%%%%%%%%%%%%%%
\section{Chaos of combinatorial type}\label{sec:3}
%%%%%%%%%%%%%%%%%%%%%%%%%%%%%%%%%%%%%%%%

While the matrix chaos inequalities of the previous section can capture a large class of models, their application may appear daunting due to the large number of flattenings that must be controlled. However, the construction of these flattenings in section \ref{sec:flattenings} by means of tensor products of canonical basis vectors $e_i$ and their transposes $e_i^\top$ suggests that the norms of the flattenings should be especially easy to control if the matrix coefficients $A_{i_1,\ldots,i_q}$ can themselves be expressed as tensor products of $e_i$ and $e_i^\top$, resulting in $\brac{0,1}$-matrices with many symmetries.

This observation naturally leads us in this section to define a special class of \emph{matrix chaoses of combinatorial type}, for which the parameters in all our matrix chaos inequalities can be computed mechanically by a simple rule. Remarkably, it turns out that many matrix chaoses that arise in theoretical computer science applications are of this special form: two important examples are graph matrices~\cite{Meka2015SumofsquaresLB,ahn2016graph} and Khatri-Rao matrices~\cite{KhatriRao1968SoTS,Kasiviswanathan2010ThePO,De2011LowerBI,Rudelson2011RowPO}. Whenever this structure is present, our methods will reduce the study of such models to a nearly trivial computation. As will be illustrated in section \ref{sec:applications}, this enables us to achieve the best known results in the literature for several applications using a unified and remarkably simple analysis.

%%%%%%%%%%%%%%%%%%%%%%%%%%%%%%%%%%%%%%%%
\subsection{Definition and guiding example}
%%%%%%%%%%%%%%%%%%%%%%%%%%%%%%%%%%%%%%%%

In order to motivate the general definition of chaos of combinatorial type, we begin by introducing the guiding example of Khatri-Rao matrices. These are random matrices with dependent entries whose study dates back to at least the 1960s~\cite{KhatriRao1968SoTS}, and have more recently been used in the context of differential privacy~\cite{Kasiviswanathan2010ThePO,De2011LowerBI}.

\begin{example}[Khatri-Rao matrices]\label{example:khatri-rao}
We begin by stating the definition.

\begin{definition}\label{def:khatri-rao}
    Let $q, n, d$ be positive integers, let $h$ be a scalar random variable with zero mean and unit variance, and let let $W^{(1)},\dots, W^{(q)}$ be $d\times n$ random matrices whose entries are i.i.d.\ copies of $h$.
    The \emph{Khatri-Rao matrix $Y$} is the $d^q\times n$ random matrix obtained by taking the column-wise Kronecker product of $W^{(1)},\dots, W^{(q)}$, that is, the matrix whose entries
    are defined by
    \begin{equation}
    \label{eq:kraodef}
        Y[(j_1,\dots,j_q),k] = \prod_{t=1}^{q} W^{(t)}[j_t,k],
    \end{equation}
    for any $j_1,\dots,j_q\in[d]$ and $k\in[n]$.
\end{definition}

The Khatri-Rao matrix \eqref{eq:kraodef} can be equivalently expressed as a decoupled matrix chaos
\begin{equation}\label{eq:KRasChaos}
        Y = \sum_{j_1,\dots,j_q\in[d], k\in[n]}  \left(\prod_{t=1}^{q}W^{(t)}[j_t,k]\right)\, e_{j_1}\otimes\cdots\otimes e_{j_q} \otimes e_k^\top
\end{equation}
with the special property that every matrix coefficient
$A_{(j_1,k),\ldots,(j_q,k)}=e_{j_1}\otimes\cdots\otimes e_{j_q} \otimes e_k^\top$ is a tensor product of coordinate basis vectors and their transposes.\footnote{Formally speaking, the definition \eqref{eq:decnongaussianchaos} of a matrix chaos requires us to assign independent indices to each coordinate of $A_{(j_1,k_1),\ldots,(j_q,k_q)}$. In order to capture \eqref{eq:kraodef}, we would then set $A_{(j_1,k_1),\ldots,(j_q,k_q)}=0$ except when $k_1=\cdots=k_q=k$. To lighten the notation, however, we will generally drop these zero coefficients from the summation as in \eqref{eq:KRasChaos}.}
\end{example}

A characteristic feature of the above example is that even though each matrix coefficient is a tensor product of coordinate basis vectors and their adjoints, the indices of these coordinate vectors may simultaneously appear in the coordinates corresponding to distinct random vectors
$\boldsymbol{h}^{(t)}$ in the definition \eqref{eq:decnongaussianchaos} of a decoupled matrix chaos.
In this example, the coordinate basis vectors that define the matrix coefficients are indexed by $j_1,\dots,j_q,k$, which we call \textbf{summation indices}. 
Each random vector $\boldsymbol{h}^{(t)}$ is indexed by an an ordered subset $(j_t,k)$ of the summation indices, which we call \textbf{chaos coordinates}. Finally, the entries of the matrix coefficients are also indexed by ordered subsets $(j_1,\ldots,j_q)$ and $k$ of the summation indices, which we call \textbf{matrix coordinates}.

The above structure is generalized by the notion of matrix chaos of combinatorial type.

\begin{definition}[Matrix Chaos of Combinatorial type]\label{def:combinatorialtype}
    Let $h$ be a scalar random variable with zero mean, let $q, p, S_1,\ldots, S_p$ be positive integers, and let $I_1,\ldots,I_{q+2}$ be ordered subsets of $[p]$.
    A \emph{matrix chaos of combinatorial type} is defined by
    \begin{equation}\label{eq:chaoscombinatorialtype}
        Y = \sum_{\s\in[S_1]\times\cdots\times[S_p]} h^{(1)}_{I_1(\s)}\cdots h^{(q)}_{I_q(\s)} \, e_{I_{q+1}(\s)} \otimes e_{I_{q+2}(\s)}^\top,
    \end{equation}
    where for
    $I=(i_1,\ldots,i_k)$ and $\s=(s_1,\ldots,s_p)$ we define $I(\s):=(s_{i_1},\ldots,s_{i_k})$,
    $$
            e_{(s_{i_1},\ldots,s_{i_k})} := e_{s_{i_1}}\otimes e_{s_{i_2}}\otimes  \cdots
            \otimes e_{s_{i_k}},
    $$
    and $h^{(t)}_{(s_{i_1},\ldots,s_{i_k})}$ are independent copies of $h$.
    Here $s_1,\dots,s_p$ are \emph{summation indices}; $I_1(\s),\dots,I_q(\s)$ are  \emph{chaos coordinates}; and $I_{q+1}(\s),I_{q+2}(\s)$ are \emph{matrix coordinates}.
\end{definition}

Further examples of chaos of combinatorial type will be treated in section~\ref{sec:applications}.

%%%%%%%%%%%%%%%%%%%%%%%%%%%%%%%%%%%%%%%%
\subsection{How to compute norms of flattenings}
%%%%%%%%%%%%%%%%%%%%%%%%%%%%%%%%%%%%%%%%

The aim of this section is to develop a user-friendly procedure to compute the norms of flattenings of chaoses of combinatorial type (Algorithm~\ref{alg:buildthetable}).

\subsubsection{The Khatri-Rao example as a warm-up} 
\label{sec:khwarmup}

We again use the guiding example of Khatri-Rao matrices to illustrate the procedure.
We focus on the case $q=2$ for simplicity.

Let $Y$ be a Khatri-Rao matrix as in \eqref{eq:KRasChaos} with $q=2$. In the notation of Definition~\ref{def:combinatorialtype} we have $q=2$ and $p=3$; the summation indices are $j_1,j_2\in[d]$ and $k\in[n]$; the chaos coordinates are given by $I_1(j_1,j_2,k)=(j_1,k)$ and $I_2(j_1,j_2,k)=(j_2,k)$, and the matrix coordinates are given by $I_3(j_1,j_2,k)=(j_1,j_2)$ and $I_4(j_1,j_2,k)=k$; and $h^{(t)}_{(j_t,k)}=W^{(t)}[j_t,k]$. We can therefore write
$$
    Y = \sum_{j_1,j_2\in[d],k\in[n]}h_{(j_1,k)}^{(1)}h_{(j_2,k)}^{(2)}\, e_{(j_1,j_2)} \otimes e_k^{\top}.
$$
For the sake of exposition, let us focus on the flattening $\flatta{1,2,3}{4}$. This is the $\sigma$-flattening where both chaos coordinates are in the row set $R$ (see~\eqref{eq:flattenings}). It is given by
\begin{eqnarray}
    \flatta{1,2,3}{4} & = & \sum_{j_1,j_2\in[d],k\in[n]} e_{(j_1,k)} \otimes e_{(j_2,k)} \otimes e_{(j_1,j_2)} \otimes e^\top_{k} \nonumber \\
    & = & \sum_{j_1,j_2\in[d],k\in[n]} e_{j_1} \otimes e_{k} \otimes e_{j_2} \otimes e_{k} \otimes e_{j_1} \otimes e_{j_2} \otimes e^\top_{k},\nonumber
\end{eqnarray}
where in the last line we used $e_{(j_1,k)} = e_{j_1} \otimes e_{k}$ (and similarly for other coordinates).

By permuting the order of tensor products (which corresponds to reordering rows and columns, and so preserves all singular values of the matrix), we obtain
\begin{eqnarray}
    \flatta{1,2,3}{4} &\simeq& \sum_{j_1,j_2\in[d],k\in[n]} e_{j_1} \otimes e_{j_1}  \otimes e_{j_2} \otimes  e_{j_2}  \otimes e_{k} \otimes e_{k} \otimes e^\top_{k} \label{eq:expanding_in_summation_indices}\\
     &=& \left(\sum_{j_1\in[d]} e_{j_1} \otimes e_{j_1} \right) \otimes \left(\sum_{j_2\in[d]}   e_{j_2} \otimes  e_{j_2} \right) \otimes\left(\sum_{k\in[n]} e_{k} \otimes e_{k} \otimes e^\top_{k} \right), \label{eq:evaluated_norms}
\end{eqnarray}
where $A\simeq B$ means that the two matrices are related by a unitary change of basis. We thus have
\begin{equation*}
    \left\|\flatta{1,2,3}{4}\right\| =
    \underbrace{\norm{\sum_{j_1\in[d]} e_{j_1} \otimes e_{j_1}}}_{=\sqrt{d}}
    \underbrace{\norm{\sum_{j_2\in[d]} e_{j_2} \otimes e_{j_2}}}_{=\sqrt{d}}
    \underbrace{\norm{\sum_{k\in[n]} e_{k} \otimes e_{k} \otimes e^\adj_{k}}}_{=1}
    = d,
\end{equation*}
where we used that $\{e_k\otimes e_k\}$ are orthonormal vectors and therefore, by a unitary change of basis and restriction to a subspace, $\sum_{j\in[d]} e_j\otimes e_j$ and $\sum_{k\in[n]} e_k\otimes e_k\otimes e_k^\top$ may be viewed as a $d$-dimensional vector of ones and an $n$-dimensional identity matrix, respectively.

Repeating this procedure for the other 11 flattenings (see Table~\ref{tab:KR2} below), we readily obtain 
\begin{equation}
\label{eq:combinatorialKR2-parameters}
    \sigma(\AA) = \max\brac{d, n^{\frac{1}{2}}}, \qquad\quad v(\AA) =  r(\AA) = d^{\frac{1}{2}}.
\end{equation}
Since $\sigma(\AA)$  dominates both  $v(\AA)$ and $r(\AA)$, Theorems~\ref{thm:iteratedrose}~and~\ref{thm:iteratedstrongrose} imply that
\begin{equation}\label{eq:combinatorialKR2-bound}
    \E\norm{Y} \mathop{\asymp_q}
    \max\brac{d, n^{\frac{1}{2}}}
\end{equation}
for $d, n \to \infty$ at any relative speed, provided that $\alpha(h)$ is sub-polynomial in $d,n$.

\subsubsection{The general case}\label{subsubsec:howtocompute}

Analogously to~\eqref{eq:expanding_in_summation_indices}, any flattening of a general chaos of combinatorial type can be written (after reordering the tensor products) as
$$
    \flatta{R}{C} \simeq \bigotimes_{u=1}^p \sum_{s_u\in[S_u]}\left(  \underbrace{e_{s_u} \otimes \cdots \otimes e_{s_u}}_{\mu_u \text{ factors}} \otimes \underbrace{e_{s_u}^{\top} \otimes \cdots \otimes e_{s_u}^{\top}}_{\nu_u \text{ factors}}\right), 
$$
where $\mu_u$ and $\nu_u$ are non-negative integers. By convention, if $\mu_u=\nu_u=0$,
the tensor product inside the brackets is to be interpreted as the scalar $1$.

The calculation now proceeds by noting, as in~\eqref{eq:evaluated_norms}, that
\begin{equation}\label{eq:mu_nu_cases}
    \left\|\sum_{s_u\in[S_u]}\left(  \underbrace{e_{s_u} \otimes \cdots \otimes e_{s_u}}_{\mu_u \text{ factors}} \otimes \underbrace{e_{s_u}^{\top} \otimes \cdots \otimes e_{s_u}^{\top}}_{\nu_u \text{ factors}}\right)\right\| = \begin{cases}
        1 &\text{if } \mu_u>0 \text{ and } \nu_u>0\\
        \sqrt{S_u} &\text{if } \mu_u > 0 \text{ xor } \nu_u > 0 \\
        S_u &\text{if } \mu_u=0 \text{ and } \nu_u=0.
    \end{cases}
\end{equation}
This can be conveniently summarized by defining by $\RR$ and $\CC$ the sets of summation indices that appear in $R$ and $C$, respectively. This yields the following result, which we prove in section~\ref{sec:proof:prop:flattenings_combinatorial}.

\begin{proposition}\label{prop:flattenings_combinatorial}
    Let $Y$ be a chaos of combinatorial type as in \eqref{eq:chaoscombinatorialtype}
    of order $q$ with $p$ summation indices. Let $R,C\subseteq[q+2]$, $\RR = \cup_{t \in R} I_t$, and $\CC = \cup_{t \in C} I_t$. Then
    \begin{equation}\label{eq:combinatorial_flattening_formula}
        \norm{\flatta{R}{C}}^2 = \brap{\prod_{u \in \RR^c}S_u} \brap{\prod_{u \in \CC^c}S_u}.
    \end{equation}
\end{proposition}

This proposition yields a straightforward algorithm to compute the norms of flattenings of chaoses of combinatorial type: given a set of choices of whether each particular chaos or matrix coordinate is in $R$ or $C$, the sets $\RR = \cup_{t \in R} I_t$ and $\CC = \cup_{t \in C} I_t$ determine which summation indices belong to row and/or column matrix coordinates, and the norm of the flattening is given by~\eqref{eq:combinatorial_flattening_formula}.

\begin{algorithm}\label{alg:buildthetable}
Construct a table with the following data:
   \begin{itemize}
    \item The flattening type: $\sigma$, $v$ or $r$;
    \item For each flattening type, list all possible assignments of the chaos
    ($I_1, \ldots, I_q$) and matrix ($I_{q+1}, I_{q+2}$) coordinates to $R$, $C$, or
    $R\cap C$.
    \item Next, for each summation index, list whether it appears in $\RR$, $\CC$, or $\RR\cap\CC$.
    \item Finally, $\norm{\flatta{R}{C}}$ can be computed directly using the formula \eqref{eq:combinatorial_flattening_formula}.\footnote{In applications, is is often the case that every summation index appears in at least one of the (chaos or matrix) coordinates, so that $\RR^c\cap\CC^c=\emptyset$. In this case, the right-hand side of \eqref{eq:combinatorial_flattening_formula} is simply the product of the dimensions of all the summation indices that appear in coordinates assigned only to $R$ or only to $C$.}
    \end{itemize}
\end{algorithm}

In Table \ref{tab:KR2}, we illustrate the application of this algorithm to the $q=2$ case of the Khatri-Rao matrix, recovering the manual computation of \eqref{eq:combinatorialKR2-parameters}.

\begin{table}[t]
    \centering
    \begin{tabular}{| c | c c : c c | c c c | c |}
        \hline
        & \multicolumn{4}{c|}{coordinates} & \multicolumn{3}{c|}{summation} & \\
        type & \multicolumn{2}{c:}{chaos} & \multicolumn{2}{c|}{matrix} & \multicolumn{3}{c|}{indices} & $\text{norm}^2$ \\
        & $j_1k$ & $j_2k$ & $j_1j_2$ & $k$ & $j_1$ & $j_2$ & $k$ & \\
        \hline
        \hline
        \multirow{4}{*}{$\sigma$}
        & \myceco \myR & \myceco \myR & \myceco \myR & \myceco \myC & \myceco \myR & \myceco \myR & \myceco \myRC & \myceco$d^2$ \\
        & \myR & \myC & \myR & \myC & \myR & \myRC & \myRC & $d$ \\
        & \myC & \myR & \myR & \myC & \myRC & \myR & \myRC & $d$ \\
        & \myceco \myC & \myceco \myC & \myceco \myR & \myceco \myC & \myceco \myRC & \myceco \myRC & \myceco \myC & \myceco$n$ \\
        \hline
        \multirow{3}{*}{$v$}
        & \myR & \myR & \myC & \myC & \myRC & \myRC & \myRC & $1$ \\
        & \myR & \myC & \myC & \myC & \myRC & \myC & \myRC & $d$ \\
        & \myC & \myR & \myC & \myC & \myC & \myRC & \myRC & $d$ \\
        \hline
        \multirow{5}{*}{$r$}
        & \myR & \myRC & \myR & \myC & \myR & \myRC & \myRC & $d$ \\
        & \myC & \myRC & \myR & \myC & \myRC & \myRC & \myRC & $1$ \\
        & \myRC & \myR & \myR & \myC & \myRC & \myR & \myRC & $d$ \\
        & \myRC & \myC & \myR & \myC & \myRC & \myRC & \myRC & $1$ \\
        & \myRC & \myRC & \myR & \myC & \myRC & \myRC & \myRC & $1$ \\
        \hline
    \end{tabular}
    \caption{Flattenings of Khatri-Rao matrices (Example \ref{example:khatri-rao}) with $q = 2$ as produced by Algorithm~\ref{alg:buildthetable}. The $\sigma$, $v$, and~$r$ parameters are the maxima of the norms of the respective flattenings. The two dominant flattenings are shaded. 
    }
    \label{tab:KR2}
\end{table}

In practice, it is generally not necessary in applications to list every possible flattening,
as one can directly analyze using \eqref{eq:combinatorial_flattening_formula} which flattenings will dominate in the matrix chaos inequalities. This will be illustrated in section~\ref{subec:kr_gen_q}, where we will analyze the Khatri-Rao model for all $q\ge 2$.

%%%%%%%%%%%%%%%%%%%%%%%%%%%%%%%%%%%%%%%%
\subsection{Chaos of nearly combinatorial type}
%%%%%%%%%%%%%%%%%%%%%%%%%%%%%%%%%%%%%%%%

It will be useful (see Sections~\ref{sec:SOS:TensorPCA}, \ref{sec:graphmatrices}, and~\ref{subsec:ellifit}) to consider a slightly more general class of chaoses that include a weight function.

\begin{definition}[Matrix Chaos of nearly Combinatorial type]\label{def:combinatorialtypekernel}
    Let $h, q, p, S_1,\ldots, S_p,I_1,\ldots,I_{q+2}$ be as in Definition~\ref{def:combinatorialtype}, and $f \colon [S_1]\times\cdots\times[S_p] \to \R$ be a weight function. A matrix chaos of nearly combinatorial type with weight function $f$ is a chaos of the form:
    \begin{equation}\label{eq:chaosnearcombinatorialtype}
        Y^{f} = \sum_{\s\in[S_1]\times\cdots\times[S_p]} f(\s)\, h^{(1)}_{I_1(\s)}\cdots h^{(q)}_{I_q(\s)} \, e_{I_{q+1}(\s)}\otimes e_{I_{q+2}(\s)}^\top.
    \end{equation}
\end{definition}

By pointwise bounding $|f|$ by its maximum $\norm{f}_\infty$, we can generalize Proposition~\ref{prop:flattenings_combinatorial} to the following bound, whose proof we defer to section~\ref{sec:proof:prop:flattenings_combinatorial}.

\begin{proposition}[Flattenings of chaoses of nearly combinatorial type]\label{prop:flattenings_nearly_combinatorial}
    Let $Y^f$ be a chaos of nearly combinatorial type as in \eqref{eq:chaosnearcombinatorialtype}
    of order $q$, $p$ summation indices, and weight function $f$.
    Let $R,C\subseteq[q+2]$, $\RR = \cup_{t \in R} I_t$, and $\CC = \cup_{t \in C} I_t$. Then
    \begin{equation}\label{eq:nearly_combinatorial_flattening_formula}
        \norm{\flatta{R}{C}^f}^2 \leq \left\|f\right\|_\infty^2 \brap{\prod_{u \in \RR^c}S_u} \brap{\prod_{u \in \CC^c}S_u}.
    \end{equation}
\end{proposition}

While the norms of flattenings could be computed exactly, Proposition~\ref{prop:flattenings_nearly_combinatorial} provides a user-friendly upper bound that enables one to directly apply Algorithm~\ref{alg:buildthetable} to chaoses of nearly combinatorial type. In sections~\ref{sec:SOS:TensorPCA}, \ref{sec:graphmatrices}, and~\ref{subsec:ellifit}, we will apply Proposition~\ref{prop:flattenings_nearly_combinatorial} to chaoses whose weight function is almost always equal to its maximum value, for which this procedure is nearly optimal.

%%%%%%%%%%%%%%%%%%%%%%%%%%%%%%%%%%%%%%%%
\section{Applications}\label{sec:applications}
%%%%%%%%%%%%%%%%%%%%%%%%%%%%%%%%%%%%%%%%

In this section we focus on four illustrative applications of our techniques. Further applications and extensions are deferred to a longer companion manuscript~\cite{IteratedNCK-Journal}.

%%%%%%%%%%%%%%%%%%%%%%%%%%%%%%%%%%%%%%%%
\subsection{Khatri-Rao matrices}\label{subec:kr_gen_q}
%%%%%%%%%%%%%%%%%%%%%%%%%%%%%%%%%%%%%%%%

Algorithm~\ref{alg:buildthetable} provides a simple recipe for computing the norms of flattenings of chaoses of combinatorial type, which can be applied manually to chaoses of small order $q$. This recipe can however also be used to reason about chaoses of arbitrary order without having to explicitly write a table for each $q$. In particular, as only the largest norm in each class of flattenings must be computed to bound  $\sigma(\mathcal{A}),v(\mathcal{A}),r(\mathcal{A})$, it suffices to analyze which choices of $R$ and $C$ minimize the number of summation indices that end up in $\RR \cap \CC$.

To illustrate this procedure, we will generalize the Khatri-Rao bound \eqref{eq:combinatorialKR2-bound} for $q=2$ to arbitrary $q \geq 2$. An analogous bound was originally derived by Rudelson~\cite[Theorem 1.3]{Rudelson2011RowPO} under more restrictive assumptions. The present bound is considerably stronger; for example, unlike the bound of \cite{Rudelson2011RowPO}, it remains valid for a large class of sparse entry distributions.

\begin{theorem}
    Let $Y$ be a Khatri-Rao random matrix as defined in Definition~\ref{def:khatri-rao}. Then
    \begin{equation*}\label{eq:KRq-bound}
        \E\norm{Y} \asymp_q \max\brac{d^{\frac{q}2}, n^{\frac12}}
    \end{equation*}
    provided that $\|h\|_{L^{q\log(d+n)}}^q \log(d+n)^{\frac{q+3}{2}} d^{\frac{q-1}{2}} =
    o(\max\{d^{\frac{q}2}, n^{\frac12}\})$.   
\end{theorem}
\begin{proof}
    For this chaos of combinatorial type, the summation indices are $j_1,\ldots,j_q$ and $k$. We claim that the following two final flattenings are dominant:\footnote{%
    For clarity of exposition, we indicate informally for $I_t$ and $\RR,\CC$ which summation indices appear in them, rather than specifying the label of the summation index as in the formal Definition \ref{def:combinatorialtype}.}
    \begin{enumerate}
        \item the $\sigma$-flattening $\flatta{1:q, q+1}{q+2}$ with all chaos coordinates $I_t = (j_t, k)$ being in $R$, has
        \begin{equation*}
            \RR = \brac{j_1,\ldots,j_q,k},~ \CC = \brac{k} \implies \|\flatta{1:q, q+1}{q+2}\|^2=d^q;
        \end{equation*}
        \item the $\sigma$-flattening $\flatta{q+1}{1:q,q+2}$ with all chaos coordinates $I_t = (j_t, k)$ being in $C$, has
        \begin{equation*}
            \RR = \brac{j_1,\ldots,j_q}, ~\CC = \brac{j_1, \ldots, j_q, k} \implies \|\flatta{q+1}{1:q,q+2}\|^2=n.
        \end{equation*}
    \end{enumerate}
    Indeed, for any other $\sigma$- or $r$-flattening $\flatta{R}{C}$, there are $t, t' \in [q]$ (possibly equal) such that $t \in R$ and $t' \in C$. Hence both summation indices $k$ and $j_{t'}$ appear in $\RR \cap \CC$, and thus $\|\flatta{R}{C}\|^2 \leq d^{q-1}$. Similarly, given an arbitrary $v$-flattening $\flatta{R}{C}$, there must be some $t \in [q] \cap R$ (as $R \neq \emptyset$), so both $k$ and $j_{t}$ appear in $\RR \cap \CC$, and thus $\|\flatta{R}{C}\|^2 \leq d^{q-1}$.

    We have therefore shown that $\sigma(\AA)=\max\{d^{\frac{q}2}, n^{\frac12}\}$ and
    that $v(\AA),r(\AA)\le d^{\frac{q-1}{2}}$. The conclusion now follows readily from
    Theorems~\ref{thm:iteratedrose}~and~\ref{thm:iteratedstrongrose}.
\end{proof}

\begin{remark}
    One of the main contributions of~\cite{Rudelson2011RowPO} is to show that the smallest singular value $s_n(Y)$ is lower bounded up to an absolute constant by $d^{\frac{q}{2}}$ whenever $n \lesssim_{q,s}  \frac{d^{q}}{\log_{(s)}(d)}$, where $\log_{(s)}(\cdot)$ is the iterated logarithm function. As will be shown in the companion paper~\cite{IteratedNCK-Journal}, a variant of our main results for the smallest singular value makes it possible to remove the $\log_{(s)}(\cdot)$ factor.
\end{remark}

%%%%%%%%%%%%%%%%%%%%%%%%%%%%%%%%%%%%%%%%
\subsection{The sum-of-squares algorithm for tensor PCA}\label{sec:SOS:TensorPCA}
%%%%%%%%%%%%%%%%%%%%%%%%%%%%%%%%%%%%%%%%

Another important example of a matrix chaos arises in the analysis of a sum-of-squares algorithm for tensor PCA~\cite{Hopkins2015TensorPC,Hopkins2018StatisticalIA}. While graph matrices (see Section~\ref{sec:graphmatrices}) are often used to provide algorithmic lower bounds, the chaos in this section is used to prove upper bounds (i.e., algorithmic guarantees). 

Hopkins and collaborators~\cite{Hopkins2018StatisticalIA} (see also~\cite[Section 6]{Hopkins2015TensorPC}) prove upper bounds on the performance of the sum-of-squares hierarchy for tensor PCA via an upper bound on the norm of $X \coloneqq \sum_{i \in [n]} \brap{W_i \otimes W_i - \E\bras{W_i \otimes W_i}}$, where $W_1, \ldots, W_n$ are i.i.d.\ $d\times d$ matrices with i.i.d.\ standard gaussian entries~(\cite[Theorem B.5]{Hopkins2015TensorPC}~and~\cite[Theorem 6.7.1 and Lemma 6.3.4]{Hopkins2018StatisticalIA}). 
Their bounds are optimal up to a logarithmic factor.
Using the methods of this paper, we can easily remove the spurious logarithmic factor in their bound.

\begin{theorem}\label{theorem:forsosnolog}
    Let $W_1, \ldots, W_n$ be i.i.d.\ $d \times d$ random matrices with i.i.d.\ $N(0,1)$ entries. Then
    \begin{equation*}
        \E\norm{\sum_{i \in [n]} \brap{W_i \otimes W_i - \E\bras{W_i \otimes W_i}}} \lesssim d\sqrt{n},
    \end{equation*}
    provided that $n,d\gtrsim \log(d+n)^4$.
\end{theorem}

Using Theorem~\ref{theorem:forsosnolog}, it is straightforward to remove the logarithmic factor in the sum-of-squares algorithmic guarantee of~\cite{Hopkins2015TensorPC,Hopkins2018StatisticalIA}. Let us note that the regime of interest in this application is $d^{\tau_{-}}\leq n\leq d^{\tau_{+}}$ for fixed $0<\tau_{-}<\tau_{+}$, so that the assumption on $n,d$ is automatically satisfied.

\begin{proof}[Proof of Theorem~\ref{theorem:forsosnolog}]
    Let $g_{i,j,k} = W_i[j,k]$. Note that $X$ naturally decomposes as
    \begin{equation*}
        X = \sum_{\substack{i \in [n]\\ j, k \in [d]}} \brap{g_{i,j,k}^2 - 1} e_{(j,j)} \otimes e_{(k,k)}^\top + \sum_{\substack{i \in [n]\\j_1, k_1, j_2, k_2 \in [d]}} \1_{(j_1, k_1) \neq (j_2, k_2)}\, g_{i,j_1,k_1} g_{i,j_2,k_2}\, e_{(j_1,j_2)} \otimes e_{(k_1,k_2)}^\top,
    \end{equation*}
    and denote by $X_1$ and $X_2$ the two terms on the right-hand side.
    \begin{enumerate}
        \item After reordering rows and columns, and using $e_j \otimes e_j \simeq e_j$, we can express $X_1$ as 
        \begin{equation*}
            Y'_1 \coloneqq \sum_{\substack{i \in [n]\\ j, k \in [d]}} \brap{g_{i,j,k}^2 - 1} e_{j} \otimes e_{k}^\top.
        \end{equation*}
        This is a chaos of combinatorial type with $h_{i,j,k}=g_{i,j,k}^2 - 1$, for which Algorithm~\ref{alg:buildthetable} outputs Table~\ref{tab:sos_tensor_PCA_first}.
        As $\alpha(h)\lesssim \log(d+nd^2)$ for $p=\log(d+nd^2)$, Theorem \ref{thm:iteratedrose} yields
        \begin{equation}\label{eq:sos_firstbound}
            \E\norm{Y'_1} \lesssim \log(d+nd^2)^{\frac{1}{2}}\sqrt{dn} + \log(d + nd^2)^2.
        \end{equation}
        \item After decoupling, $X_2$ corresponds to the chaos
        \begin{equation*}
            Y_2 \coloneqq \sum_{\substack{i \in [n]\\j_1, k_1, j_2, k_2 \in [d]}} \underbrace{\1_{(j_1, k_1) \neq (j_2, k_2)}}_{\text{weight function }f}\, g_{i,j_1,k_1}^{(1)} g_{i,j_2,k_2}^{(2)}\, e_{(j_1,j_2)} \otimes e_{(k_1,k_2)}^\top.
        \end{equation*}
        This is a chaos of nearly combinatorial type. We can therefore use Proposition~\ref{prop:flattenings_nearly_combinatorial} to to upper bound the parameters by the output of Algorithm~\ref{alg:buildthetable}, which is given in Table~\ref{tab:sos_tensor_PCA_second}. The iterated strong NCK inequality (Theorem~\ref{thm:iteratedfreenck}) yields
        \begin{equation}\label{eq:sos_secondbound}
            \E\norm{Y_2} \lesssim d\sqrt{n} + \log(d^2+nd^2)^{2} \brap{d \lor \sqrt{n}}.
        \end{equation}
    \end{enumerate}
    Combining~\eqref{eq:sos_firstbound} and~\eqref{eq:sos_secondbound} gives the desired bound.\qedhere
\end{proof}

\begin{table}
    \centering
    \begin{tabular}{| c | c : c c | c c c | c |}
        \hline
        & \multicolumn{3}{c|}{coordinates} & \multicolumn{3}{c|}{summation} & \\
        type & \multicolumn{1}{c:}{chaos} & \multicolumn{2}{c|}{matrix} & \multicolumn{3}{c|}{indices} & $\text{norm}^2$ \\
        & $ijk$ & $j$ & $k$ & $i$ & $j$ & $k$ & \\
        \hline
        \hline
        \multirow{2}{*}{$\sigma$}
        & \myceco \myR & \myceco \myR & \myceco \myC &  \myceco \myR & \myceco \myR & \myceco \myRC  & \myceco $n d$ \\
        & \myceco \myC &  \myceco \myR & \myceco \myC &  \myceco \myC &  \myceco \myRC  & \myceco \myC &  \myceco $n d$ \\
        \hline
        \multirow{1}{*}{$r$}
        & \myRC & \myR &  \myC &  \myRC & \myRC  & \myRC  & $1$ \\
        \hline
    \end{tabular}
    \caption{Flattenings of $\AA'_1$: $\sigma(\AA'_1) = \sqrt{nd}$, $r(\AA'_1) = 1$, used in~\eqref{eq:sos_firstbound}.}
    \label{tab:sos_tensor_PCA_first}
\end{table}

\begin{table}
    \centering
    \begin{tabular}{| c | c c : c c | c c c c c | c |}
        \hline
        & \multicolumn{4}{c|}{coordinates} & \multicolumn{5}{c|}{summation} & \\
        type & \multicolumn{2}{c:}{chaos} & \multicolumn{2}{c|}{matrix} & \multicolumn{5}{c|}{indices} & $\text{norm}^2$ \\
        & $ij_1k_1$ & $ij_2k_2$ & $j_1j_2$ & $k_1k_2$ & $i$ & $j_1$ & $j_2$ & $k_1$ & $k_2$ & \\
        \hline
        \hline
        \multirow{4}{*}{$\sigma$}
        & \myceco \myR & \myceco \myR & \myceco \myR & \myceco \myC &  \myceco \myR & \myceco \myR & \myceco \myR & \myceco \myRC  & \myceco \myRC  & \myceco $n d^2$ \\
        & \myR & \myC &  \myR & \myC &  \myRC  & \myR & \myRC  & \myRC  & \myC &  $d^2$ \\
        & \myC &  \myR & \myR & \myC &  \myRC  & \myRC  & \myR & \myC &  \myRC  & $d^2$ \\
        & \myceco \myC &  \myceco \myC &  \myceco \myR & \myceco \myC &  \myceco \myC &  \myceco \myRC  & \myceco \myRC  & \myceco \myC &  \myceco \myC &  \myceco $n d^2$ \\
        \hline
        \multirow{3}{*}{$v$}
        & \myR & \myR & \myC &  \myC &  \myR & \myRC  & \myRC  & \myRC  & \myRC  & $n$ \\
        & \myR & \myC &  \myC &  \myC &  \myRC  & \myRC  & \myC &  \myRC  & \myC &  $d^2$ \\
        & \myC &  \myR & \myC &  \myC &  \myRC  & \myC &  \myRC  & \myC &  \myRC  & $d^2$ \\
        \hline
    \end{tabular}
    \caption{Flattenings of $\AA_2$: $\sigma(\AA_2) = d\sqrt{n}$, $v(\AA_2) = d \lor \sqrt{n}$, used in~\eqref{eq:sos_secondbound}.}
    \label{tab:sos_tensor_PCA_second}
\end{table}

We note that the same random matrix, and an analogous bound, also appears in work on quantum expanders~\cite[Theorem 1]{Lancien2023ANO}. Our aim here is to illustrate that we readily recover the correct bound by a mechanical application of matrix chaos inequalities.

%%%%%%%%%%%%%%%%%%%%%%%%%%%%%%%%%%%%%%%%
\subsection{Graph matrices}\label{sec:graphmatrices}
%%%%%%%%%%%%%%%%%%%%%%%%%%%%%%%%%%%%%%%%

The standard framework~\cite{Potechin2020MachineryFP} for obtaining algorithmic lower bounds in the sum-of-squares hierarchy is to construct a candidate pseudo-expectation, and to show that its moment matrix is positive semidefinite. When providing lower bounds for average case instances, a now standard way to construct candidate pseudo-expectation matrices is through matrix chaoses. A major challenge in this area has been that most classical random matrix inequalities were not able to analyze the spectrum of these chaoses. 

This bottleneck was resolved~\cite{Meka2015SumofsquaresLB,Barak2016ANT} by the development of a theory of the so-called \emph{graph matrices}~\cite{Medarametla2016BoundsOT,ahn2016graph}. 
One can think of these as a natural basis in which moment matrices (at least those that possess ``enough symmetry'') can be expressed. For any graph matrix, norm bounds are known~\cite{ahn2016graph} (see section \ref{sec:graphnormbounds} below) which in certain cases translate to bounds for moment matrices. This approach is used in showing several of the state of the art lower bounds for average case complexity in the sum-of-squares hierarchy, see~\cite{Potechin2020MachineryFP,ahn2016graph}.

\subsubsection{Definition}
Graph matrices are random matrices that depend on an \emph{input distribution} of ${\binom{n}{2}}$ i.i.d.\ Rademacher random variables (each corresponding to an edge of a complete graph on $n$ nodes), and a small sized graph $\alpha$ called a \emph{shape}, with identified subsets $U_\alpha, V_\alpha \subseteq V(\alpha)$ of, respectively, left and right vertices.
The shape will be fixed, while $n$ is best thought of as arbitrarily large (in other words, we will not aim to optimize the dependency of our bonds on constants depending on $\alpha$).

\begin{definition}[Shape]\label{def:shape}
    A \emph{shape} is a graph, that has a subset $U_\alpha \subseteq V(\alpha)$ of \emph{left vertices}, and another subset $V_\alpha \subseteq V(\alpha)$ of \emph{right vertices}.\footnote{The sets $U_\alpha$ and $V_\alpha$ can intersect, their union is not necessarily $V(\alpha)$, and their sizes are not necessarily equal.}
\end{definition}

\begin{definition}[Graph matrices]\label{def:graphmatrices}
    Let $\alpha$ be a shape and $n$ be a large integer.
    \begin{enumerate}
        \item The set of \emph{middle vertices} is given by $W_\alpha = V(\alpha) \setminus (U_\alpha \cup V_\alpha)$.

        \item The \emph{ground set} is the set of indices $[n] = \brac{1, \ldots, n}$, which we also interpret as the vertices of the complete graph $K_n$.
        
        \item The \emph{input distribution} $\vect{\eps} = \brap{\eps_e}_{e \in E(K_n)}$ is a collection of i.i.d. Rademachers indexed by the edges of $K_n$ (i.e. unordered pairs of distinct numbers).

        \item A \emph{realization} is any injective map $\varphi \colon\! V(\alpha) \to [n]$ from the shape vertices to the ground set.
        
        \item The \emph{graph matrix} $M_\alpha$ is the $n^{\abs{U_\alpha}} \times n^{\abs{V_\alpha}}$ random matrix, whose rows and columns are indexed by ordered subsets of $[n]$ with cardinality $\abs{U_\alpha}$ and $\abs{V_\alpha}$, respectively, given by
        \begin{equation}\label{eq:graphmatrixdef}
            M_\alpha \coloneqq \sum_{\text{realization } \varphi} \brap{\prod_{(i,j) \in E(\alpha)} \eps_{\varphi(i),\varphi(j)}} e_{\varphi(U_\alpha)}\otimes e_{\varphi(V_\alpha)}^\top.
        \end{equation}
    \end{enumerate}
\end{definition}

\begin{remark}[Identification in notation]\label{rmk:identification}
    In the discussion that follows, we will treat graph matrices within the framework developed in Section~\ref{sec:3}. Observe that summing over realizations in~\eqref{eq:graphmatrixdef} corresponds to different choices for summation indices in~\eqref{eq:chaoscombinatorialtype}. Thus we will, in a slight abuse of notation, use the same symbols to denote vertices from $V(\alpha)$ and summation indices.
\end{remark}

\begin{example}[Examples of graph matrices]\label{ex:graphmatrices}
    These examples are also represented in Figure~\ref{fig:graphmatrices}.
    \begin{enumerate}
    \itemsep\medskipamount
        \item (Wigner without a diagonal) If $U_\beta = \brac{i}, V_\beta = \brac{j}$, $W_\beta = \emptyset$, $E(\beta) = \brac{(i,j)}$, then
        \begin{equation}\label{eq:gmwigner}
            M_\beta = \sum_{i \neq j} \eps_{i,j} \, e_i \otimes e_j^\top
        \end{equation}
        is an $n \times n$ Wigner matrix with zeros on the diagonal.
        
        \item (Z--shaped graph matrix) If $U_\gamma = \brac{i, j}, V_\gamma = \brac{k, l}$, $W_\gamma = \brac{\emptyset}$, $E(\gamma) = \brac{(i,k),(j,k),(j,l)}$, 
        \begin{equation*}
            M_\gamma = \sum_{i, j, k, l\text{ distinct}} \eps_{i,k}\eps_{j,k}\eps_{j,l} \, e_{(i,j)} \otimes e_{(k,l)}^\top
        \end{equation*}
        is an $n^2 \times n^2$ asymmetric matrix that was studied in the context of free probability~\cite{Cai2022OnMD}.
        
        \item (Example of a graph matrix with middle vertices) If $U_\delta = \brac{i, j}, V_\delta = \brac{k, l}$, $W_\delta = \brac{m, o}$, $E(\delta) = \brac{(i,m),(j,m),(k,m),(l,m)}$, then
        \begin{equation*}
            M_\delta = \sum_{i, j, k, l, m, o\text{ distinct}} \eps_{i,m}\eps_{j,m}\eps_{k,m}\eps_{l,m} \, e_{(i,j)} \otimes e_{(k,l)}^\top
        \end{equation*}
        is an $n^2 \times n^2$ symmetric matrix. Note that $\abs{W_\delta}$ has no effect on the dimension of $M_\delta$.
    \end{enumerate}
\end{example}

\subsubsection{Norm bounds}\label{sec:graphnormbounds}

We are ready to state a general bound on the norm of graph matrices.
Recall that a set of vertices $S$ is a $U$---\,$V$ vertex separator if all paths from $U$ to $V$ pass through $S$.

\begin{theorem}[Graph matrix norm bounds]\label{thm:gmnormbound}
    Given a shape $\alpha$, let $M_\alpha$ be the associated graph matrix as in~\eqref{eq:graphmatrixdef}. Then we have
    \begin{equation}\label{eq:gmbound}
        n^{\frac{1}{2}\brap{\abs{V(\alpha)}-\abs{\smin}+\abs{\wiso}}}
        \lesssim_{\alpha} \E\norm{M_\alpha}
        \lesssim_{\alpha} n^{\frac{1}{2}\brap{\abs{V(\alpha)}-\abs{\smin}+\abs{\wiso}}} \cdot (\log n)^{\frac{1}{2}f(\alpha)},
    \end{equation}
    where $f(\alpha) = \abs{\smin} - \abs{U_\alpha \cap V_\alpha} + \abs{W_\alpha} - \abs{\wiso}$. Here $\abs{\smin}$ is the size of the minimal $U_\alpha$---\,$V_\alpha$ vertex separator and $\wiso$ is the set of all isolated middle vertices.
\end{theorem}

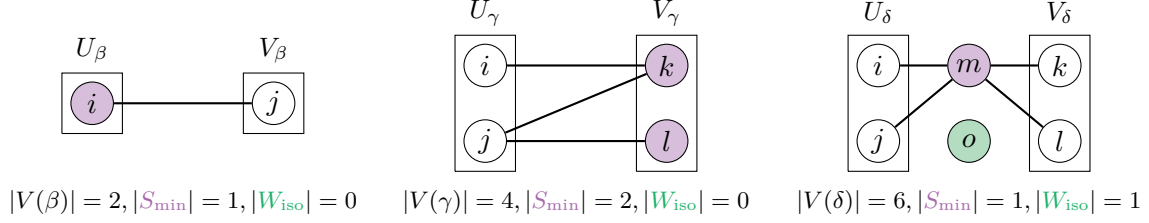
\begin{figure}[t]
    \centering
    \begin{tikzpicture}
        \def\xoffeset{1.2};
        \def\voffeset{0.5};
        \def\vequation{-1.3};
        \def\rsep{0.4};
        
        \def\x{-5.2};
        \myrectangle{\x-\xoffeset}{0.4}{0}{0.4}{\small$U_\beta$}
        \myrectangle{\x+\xoffeset}{0.4}{0}{0.4}{\small$V_\beta$}
        \mynode{\x-\xoffeset}{0}{i1}{$i$}{fill=\mypurple}
        \mynode{\x+\xoffeset}{0}{j1}{$j$}{}
        \myedge{i1}{j1}
        \mygmbound{\x}{\vequation}{\beta}{2}{1}{0}

        \def\x{0};
        \myrectangle{\x-\xoffeset}{0.4}{0}{0.9}{\small$U_\gamma$}
        \myrectangle{\x+\xoffeset}{0.4}{0}{0.9}{\small$V_\gamma$}
        \mynode{\x-\xoffeset}{\voffeset}{i2}{$i$}{}
        \mynode{\x-\xoffeset}{-\voffeset}{j2}{$j$}{}
        \mynode{\x+\xoffeset}{\voffeset}{k2}{$k$}{fill=\mypurple}
        \mynode{\x+\xoffeset}{-\voffeset}{l2}{$l$}{fill=\mypurple}
        \myedge{i2}{k2}
        \myedge{j2}{k2}
        \myedge{j2}{l2}
        \mygmbound{\x}{\vequation}{\gamma}{4}{2}{0}
        
        \def\x{5.2};
        \myrectangle{\x-\xoffeset}{0.4}{0}{0.9}{\small$U_\delta$}
        \myrectangle{\x+\xoffeset}{0.4}{0}{0.9}{\small$V_\delta$}
        \mynode{\x-\xoffeset}{\voffeset}{i3}{$i$}{}
        \mynode{\x-\xoffeset}{-\voffeset}{j3}{$j$}{}
        \mynode{\x+\xoffeset}{\voffeset}{k3}{$k$}{}
        \mynode{\x+\xoffeset}{-\voffeset}{l3}{$l$}{}
        \mynode{\x}{\voffeset}{m3}{$m$}{fill=\mypurple}
        \mynode{\x}{-\voffeset}{o3}{$o$}{fill=\mygreen}
        \myedge{i3}{m3}
        \myedge{j3}{m3}
        \myedge{k3}{m3}
        \myedge{l3}{m3}
        \mygmbound{\x}{\vequation}{\delta}{6}{1}{1}
    \end{tikzpicture}
    \caption{Using Theorem~\ref{thm:gmnormbound} on graph matrices from Example~\ref{ex:graphmatrices} yields the following bounds (logarithmic factors omited): $\norm{M_\beta} \approx \sqrt{n}$, $\norm{M_\gamma} \approx n$, $\norm{M_\delta} \approx n^3$.}
    \label{fig:graphmatrices}
\end{figure}

The upper bound in Theorem~\ref{thm:gmnormbound} first appeared in~\cite{ahn2016graph,Medarametla2016BoundsOT} where an intricate moment method argument is used. The minimal vertex separator $\smin$ appears indirectly as a consequence of duality between max-flow and min-cut. A lower bound was shown there for most shapes. More recently, a similar upper bound was obtained in \cite{rajendran2023concentration} by analyzing matrices of partial derivatives that arise by iterating Efron-Stein inequalities. In the case of Rademachers, these matrices are deterministic (and 
coincide with the flattenings discussed here), and $\smin$ naturally arises in computations of Frobenius norms. This provides a more direct proof of the upper bound, but with a larger power of the logarithm: an edge quantity $\abs{E(\alpha)}$ replaces the vertex quantity $f(\alpha)$.

Our tools allow a direct proof of Theorem~\ref{thm:gmnormbound} with the $f(\alpha)$ logarithmic power and provide a lower bound for all shapes.
In this manuscript we provide a proof of the upper bound, whereas a proof of the lower bound is deferred to~\cite{IteratedNCK-Journal} (see Remark~\ref{remark:lowerboundgraphmatrices}).

It should be emphasized, however, that the proof of Theorem~\ref{thm:gmnormbound} only uses the simplest iterated NCK inequality to achieve universal bounds. The main benefit of our framework is that it provides an effortless way to remove logarithmic factors in instances where $v$-flattenings are negligible, by using instead the iterated strong inequalities.
The latter provides a systematic method for achieving improved bounds for graph matrices, as will be illustrated in
Section~\ref{subsec:ellifit} below.

\subsubsection{Flattenings of graph matrices}

One small obstacle to directly proving Theorem~\ref{thm:gmnormbound} using Proposition~\ref{prop:flattenings_nearly_combinatorial} is the fact that the input distribution is indexed by edges, not by ordered pairs---in other words $\eps_{\varphi(i),\varphi(j)}\equiv\eps_{\varphi(j),\varphi(i)}$. This obstacle is already present when trying to upper bound the spectral norm in the example of $M_\beta$ defined in~\eqref{eq:gmwigner}. However, if we decompose
\begin{equation*}
    M_\beta = \sum_{i, j} \eps_{i,j} \,e_i\otimes e_j^\top = \sum_{i, j} \1_{i < j}\, \eps_{i,j} \,e_i\otimes e_j^\top + \sum_{i, j} \1_{i > j}\, \eps_{i,j} \,e_i\otimes e_j^\top,
\end{equation*}
then each summand is a chaos of nearly combinatorial type, and its parameters can be analyzed with Proposition~\ref{prop:flattenings_nearly_combinatorial}.
We will apply a similar idea in the general setting.

\begin{proof}[Proof of Theorem~\ref{thm:gmnormbound}: upper bound with $(\log n)^{\frac{1}{2}E(\alpha)}$ factor]
    Given a graph matrix $M_\alpha$, we have
    \begin{align*}
        M_\alpha
        &= \sum_{\text{realization } \varphi} \brap{\prod_{(i,j) \in E(\alpha)} \eps_{\varphi(i),\varphi(j)}} e_{\varphi(U_\alpha)}\otimes e_{\varphi(V_\alpha)}^\top\\
        &= \sum_{E\subseteq E(\alpha)}\sum_{\varphi} 
        \brap{\prod_{(i,j) \in E}\1_{\varphi(i) > \varphi(j)}
        \prod_{(i,j) \in E(\alpha)\backslash E}\1_{\varphi(i) < \varphi(j)}}
        \brap{\prod_{(i,j) \in E(\alpha)} \eps_{\varphi(i),\varphi(j)}}
        e_{\varphi(U_\alpha)}\otimes e_{\varphi(V_\alpha)}^\top.
    \end{align*}
    The summand $M_{\alpha, E}$ associated to each (possibly empty) subset of edges $E\subseteq E(\alpha)$ is, after decoupling, is a chaos of nearly combinatorial type (Definition~\ref{def:combinatorialtypekernel}). Each chaos has  
    \begin{itemize}
        \item $p = \abs{V(\alpha)}$ summation indices $(s_v)_{v\in V(\alpha)}$ (which correspond to $s_v := \varphi(v)$, see Remark~\ref{rmk:identification});
        \item $q = \abs{E(\alpha)}$ chaos coordinates, which correspond to shape edges:
        \begin{equation*}
            I_e(\s) = \brap{s_u, s_v} \text{ for } e=(u,v) \in E(\alpha);
        \end{equation*}
        \item matrix coordinates given by
        \begin{equation*}
            I_{q+1}(\s) = (s_u)_{u \in U_\alpha}, \qquad I_{q+2}(\s) = (s_u)_{u \in V_\alpha};
        \end{equation*}
        \item a weight function whose $\ell_\infty$ norm is $1$.
    \end{itemize}
    Consider any final $\sigma$-flattening $\flatta{R}{C}$ of $M_{\alpha,E}$. Then the  formula~\eqref{eq:nearly_combinatorial_flattening_formula} yields
    \begin{equation*}
        \norm{\flatta{R}{C}} \leq n^{\frac{1}{2}\abs{\RR^c}} n^{\frac{1}{2}\abs{\CC^c}} = n^{\frac{1}{2}\brap{\abs{V(\alpha)} - \abs{\RR \cap \CC} + \abs{\RR^c \cap \CC^c}}}.
    \end{equation*}
    The following two key inequalities explain the polynomial power in~\eqref{eq:gmbound}:
    \begin{enumerate}
    \itemsep\medskipamount
        \item $\abs{\RR \cap \CC} \geq \abs{\smin}$ holds as vertices in $\RR \cap \CC$ form a vertex separator between $U_\alpha$ and $V_\alpha$: indeed, any path in $\alpha$ that starts in $U_\alpha \subseteq \RR$ and ends in $V_\alpha \subseteq \CC$ has a vertex in $\RR \cap \CC$. The equality $\abs{\RR \cap \CC} = \abs{\smin}$ is attained whenever $R$ consists precisely of all edges that are accessible from $U_\alpha$ without passing through $\smin$.
        
        \item $\abs{\RR^c \cap \CC^c} \leq \abs{\wiso}$ holds as summation indices that do not appear in $\RR$ nor $\CC$ must correspond to isolated middle vertices, as they do not have an incident edge (in $I_e$ for some $e\in E(\alpha)$) and do not appear on the left or right sides of the shape (in $I_{q+1}$ or $I_{q+2}$).
    \end{enumerate}
    Thus
    \begin{equation}\label{eq:gmsigmaparameter}
        \sigma(\AA) \leq n^{\frac{1}{2}\brap{\abs{V(\alpha)} - \abs{\smin} + \abs{\wiso}}},
    \end{equation}
    and an upper bound as in~\eqref{eq:gmbound} with the multiplicative factor $\log(n)^{\frac{1}{2}\abs{E(\alpha)}}$ follows by using the iterated NCK inequality (Theorem~\ref{thm:iteratednck}) and the triangle inequality over all $2^q$ choices of $E$.
\end{proof}

\subsubsection{Intermediate flattenings of graph matrices}

We now focus our attention on improving the logarithmic factor. 
Recall from Section~\ref{sec:iteration} that iterating the NCK inequality yields a bound on the norm of a matrix chaos in terms of its intermediate flattenings. More precisely, 
after performing $k \le q$ iterations of the NCK inequality, one obtains a \emph{partially iterated NCK inequality}:
\begin{equation}\label{eq:stoppednck}
    \E\norm{Y} \qlesssim \log(d+m)^{\frac{k}{2}} 
    \max_{R' \sqcup C' = \brac{q-k+1, \ldots, q}} \E\norm{\interflatty{1:q-k}{R'\cup\brac{q+1}}{C'\cup\brac{q+2}}}.
\end{equation}
When $k=q$, this reduces to the iterated NCK inequality of Theorem \ref{thm:iteratednck}.

In the present setting, however, it will be useful to apply this bound with $k<q$.
The reason is that when the random variables $\boldsymbol{h}^{(t)}$ in the matrix chaos are uniformly bounded (as is the case for the Rademacher variables that appear here), we can upper bound $\interflatty{Z}{R}{C}$ \emph{entrywise} to recover a regular flattening whose norm can be computed using the formula \eqref{eq:nearly_combinatorial_flattening_formula} (see Remark~\ref{rmk:interflatt_nearcombinatorial}). We will show that the chaos variables of graph matrices can always be ordered so that $k\le f(\alpha)$ iterations suffice to achieve the same upper bound on the partial flattenings as was obtained in the previous section for the final flattenings, resulting in an improved power of the logarithm.

\begin{proof}[Proof of Theorem~\ref{thm:gmnormbound}: upper bound with $(\log n)^{\frac{1}{2}f(\alpha)}$ factor]
We begin by choosing a special ordering of the edges $E(\alpha)$ of the given shape $\alpha$, as follows.
    \begin{enumerate}
    \itemsep\medskipamount
        \item By Menger's theorem (Theorem~\ref{thm:menger}), there is a family of $\abs{\smin}$ vertex-disjoint paths from $U_\alpha$ to $V_\alpha$, each of which contains exactly one point from $U_\alpha$ and one point from $V_\alpha$. 
        We place the union of all $k_1$ edges in these paths last in our ordering of $E(\alpha)$.
        \item Next, we choose the smallest number $k_2$ of additional edges,
        so that every non-isolated middle vertex that is not contained in one of the above paths is incident to one of the additional edges. We place the additional edges in 
        the middle of our ordering of $E(\alpha)$.
        \item All remaining edges are placed at the beginning of our ordering of $E(\alpha)$.
    \end{enumerate}
We claim that $k=k_1+k_2 \leq f(\alpha)$.
Indeed, by construction, the set of paths constructed in the first step contains exactly $\abs{\smin}-\abs{U_\alpha\cap V_\alpha}$ paths of lengths $\ell_i\ge 2$, each of which contains exactly $\ell_i-1$ edges and $\ell_i-2$ middle vertices. The union of these paths therefore contain exactly 
$$
    \sum_{i=1}^{\abs{\smin}-\abs{U_\alpha\cap V_\alpha}} (\ell_i-2) =
    k_1-\abs{\smin}+\abs{U_\alpha\cap V_\alpha}    
$$
(necessarily non-isolated) middle vertices. As the total number of non-isolated middle vertices is $\abs{W_\alpha} - \abs{\wiso}$, we must therefore choose at most
$$
    k_2 \le \abs{W_\alpha} - \abs{\wiso} - 
    (k_1-\abs{\smin}+\abs{U_\alpha\cap V_\alpha}) = f(\alpha)-k_1
$$
additional edges in the second step. This establishes the claim.
    
Now let $M_{\alpha, E}$ be as in the proof of Theorem~\ref{thm:gmnormbound}, and let $Y$ be its decoupled version which is a chaos of nearly combinatorial type. 
Then any intermediate flattening $\interflatta{Z}{R}{C}$ that appears in~\eqref{eq:stoppednck} has the last $k$ shape edges (chaos coordinates) assigned to either $R$ or $C$. Therefore:
    \begin{enumerate}
    \itemsep\medskipamount
        \item Each path constructed in the first step above contains at least one vertex (summation index) in $\RR\cap\CC$, so $\abs{\RR \cap \CC} \geq \abs{\smin}$;
        \item every non-isolated middle vertex is in $\RR \cup \CC$, so $\abs{\RR^c \cap \CC^c} \leq \abs{\wiso}$.
    \end{enumerate}
By upper bounding $\interflatty{Z}{R}{C}$ entrywise and applying \eqref{eq:nearly_combinatorial_flattening_formula} (see Remark~\ref{rmk:interflatt_nearcombinatorial}), we obtain
    \begin{equation*}
        \max_{R' \sqcup C' = \brac{q-k+1, \ldots, q}} \E\norm{\interflatty{1:q-k}{R'\cup\brac{q+1}}{C'\cup\brac{q+2}}} \leq n^{\frac{1}{2}\brap{\abs{V(\alpha)} - \abs{\smin} + \abs{\wiso}}}
    \end{equation*}
precisely as in~\eqref{eq:gmsigmaparameter}. The conclusion now follows from the partially iterated NCK inequality~\eqref{eq:stoppednck}.
\end{proof}

\begin{remark}\label{remark:lowerboundgraphmatrices}
    The lower bound in Theorem~\ref{thm:gmnormbound} can be proved by considering a chaos of combinatorial type that is obtained from $M_\alpha$ by considering only a subset of the summands (by restricting the input distribution to the edge set of a $V(\alpha)$-partite graph on $n$ nodes). We defer the details of this argument to~\cite{IteratedNCK-Journal}; a similar idea is used in~\cite{ahn2016graph}.
\end{remark}

%%%%%%%%%%%%%%%%%%%%%%%%%%%%%%%%%%%%%%%%
\subsection{Sharper bounds on graph matrices and ellipsoid fitting}\label{subsec:ellifit}
%%%%%%%%%%%%%%%%%%%%%%%%%%%%%%%%%%%%%%%%

An important example where a sharper bound on the spectral norm of a graph matrix was derived is in the context of the ellipsoid fitting problem~\cite{Hsieh2023EllipsoidFU}. The ellipsoid fitting conjecture is a question in stochastic geometry that has received considerable attention recently (see~\cite{tulsiani2023ellipsoid,Hsieh2023EllipsoidFU,bandeira2024fitting} and references therein). In order to obtain a lower bound of the correct asymptotic order,\footnote{A lower bound with the correct asymptotic order was concurrently obtained in~\cite{tulsiani2023ellipsoid,Hsieh2023EllipsoidFU,bandeira2024fitting}.} the authors of~\cite{Hsieh2023EllipsoidFU} developed techniques to remove spurious logarithmic factors from the bound on the spectrum of certain graph matrices. These arguments involve sophisticated refinements of moment method calculations. In this section we show how Theorem~\ref{thm:iteratedstrongrose} and Algorithm~\ref{alg:buildthetable} can be used to effortlessly recover these improvements as a mechanical application of our general theory.

The two random matrices that need to be analyzed in this procedure (we refer the reader to~\cite{Hsieh2023EllipsoidFU}, in particular Proposition 2.3 in this reference, for the derivation of how these matrices arise) are the $m\times m$ random matrices $M_\phi$ and $M_\psi$ given by
\begin{equation}\label{eq:Malpha}
    M_\phi = \sum_{i\neq j\in [m]}\sum_{a \neq b \in [d]} \brap{g_{i,a}g_{i,b}g_{j,a}g_{j,b}}  e_i \otimes e_j^\top,
\end{equation}
\begin{equation}\label{eq:Mbeta}
    M_\psi = \sum_{i\neq j\in [m]}\sum_{a \in [d]} \brap{g_{i,a}^2-1}\brap{g_{j,a}^2-1} 
    e_i\otimes e_j^\top,
\end{equation}
where $\brap{g_{i,a}}_{i \in [m], a \in [d]}$ are $md$ i.i.d.\ standard gaussian variables. The motivating example has $m\asymp d^2$, see~\cite{Hsieh2023EllipsoidFU}, so that the assumption of Theorem \ref{thm:ellips} below is automatically satisfied.

\begin{remark}
While $M_\phi$ and $M_\psi$ are not precisely graph matrices in the sense of Definition \ref{def:graphmatrices}, they may be viewed as generalized graph matrices in the sense of \cite{ahn2016graph}. Here we gloss over the distinction and simply view these matrices
as special instances of chaoses of combinatorial type.
\end{remark}

Using our tools we provide an alternative proof of Lemma 2.7 from~\cite{Hsieh2023EllipsoidFU} (note that there is an additional scaling by $d^2$ to obtain random variables of unit variance).

\begin{theorem}\label{thm:ellips}
    Let $M_\phi$ and $M_\psi$ be the random matrices in~\eqref{eq:Malpha} and~\eqref{eq:Mbeta}. We have
    \begin{equation*}
        \E\norm{M_\phi} \lesssim d\sqrt{m} \lor m,
        \qquad
        \E\norm{M_\psi} \lesssim m \lor \sqrt{md},
    \end{equation*}
    provided that $d, m \gtrsim \log(d+m)^{9}$.
\end{theorem}

\begin{proof}
The proof is similar to that of Theorem \ref{theorem:forsosnolog}.
Note that $M_\phi$ and $M_\psi$ are square-free matrix chaoses, whose decoupled versions are chaoses of combinatorial type.

\begin{table}[t]
%\centering
\begin{minipage}[t]{.49\linewidth}
{\tiny
    \begin{tabular}[t]{| c | c c c c : c c | c c c c | c |}
        \hline
        & \multicolumn{6}{c|}{coordinates} & \multicolumn{4}{c|}{summation} & \\
        tp. & \multicolumn{4}{c:}{chaos} & \multicolumn{2}{c|}{matrix} & \multicolumn{4}{c|}{indices} & $\text{norm}^2$ \\
        & $ia$ & $ib$ & $ja$ & $jb$ & $i$ & $j$ & $i$ & $j$ & $a$ & $b$ & \\
        \hline
        \hline
        \multirow{16}{*}{\small $\sigma$}
        & \myceco \myR & \myceco \myR & \myceco \myR & \myceco \myR & \myceco \myR & \myceco \myC &  \myceco \myR & \myceco \myRC  & \myceco \myR & \myceco \myR & \myceco $m d^2$ \\
        & \myR & \myR & \myR & \myC &  \myR & \myC &  \myR & \myRC  & \myR & \myRC  & $m d$ \\
        & \myR & \myR & \myC &  \myR & \myR & \myC &  \myR & \myRC  & \myRC  & \myR & $m d$ \\
        & \myceco \myR & \myceco \myR & \myceco \myC &  \myceco \myC &  \myceco \myR & \myceco \myC &  \myceco \myR & \myceco \myC &  \myceco \myRC  & \myceco \myRC  & \myceco $m^2$ \\
        & \myR & \myC &  \myR & \myR & \myR & \myC &  \myRC  & \myRC  & \myR & \myRC  & $d$ \\
        & \myR & \myC &  \myR & \myC &  \myR & \myC &  \myRC  & \myRC  & \myR & \myC &  $d^2$ \\
        & \myR & \myC &  \myC &  \myR & \myR & \myC &  \myRC  & \myRC  & \myRC  & \myRC  & $1$ \\
        & \myR & \myC &  \myC &  \myC &  \myR & \myC &  \myRC  & \myC &  \myRC  & \myC &  $m d$ \\
        & \myC &  \myR & \myR & \myR & \myR & \myC &  \myRC  & \myRC  & \myRC  & \myR & $d$ \\
        & \myC &  \myR & \myR & \myC &  \myR & \myC &  \myRC  & \myRC  & \myRC  & \myRC  & $1$ \\
        & \myC &  \myR & \myC &  \myR & \myR & \myC &  \myRC  & \myRC  & \myC &  \myR & $d^2$ \\
        & \myC &  \myR & \myC &  \myC &  \myR & \myC &  \myRC  & \myC &  \myC &  \myRC  & $m d$ \\
        & \myC &  \myC &  \myR & \myR & \myR & \myC &  \myRC  & \myRC  & \myRC  & \myRC  & $1$ \\
        & \myC &  \myC &  \myR & \myC &  \myR & \myC &  \myRC  & \myRC  & \myRC  & \myC &  $d$ \\
        & \myC &  \myC &  \myC &  \myR & \myR & \myC &  \myRC  & \myRC  & \myC &  \myRC  & $d$ \\
        & \myceco \myC &  \myceco \myC &  \myceco \myC &  \myceco \myC &  \myceco \myR & \myceco \myC &  \myceco \myRC  & \myceco \myC &  \myceco \myC &  \myceco \myC &  \myceco $m d^2$ \\
        \hline
    \end{tabular}
}
\end{minipage}
\hfill
\begin{minipage}[t]{.49\linewidth}
{\tiny
    \begin{tabular}[t]{| c | c c c c : c c | c c c c | c |}
        \hline
        & \multicolumn{6}{c|}{coordinates} & \multicolumn{4}{c|}{summation} & \\
        tp. & \multicolumn{4}{c:}{chaos} & \multicolumn{2}{c|}{matrix} & \multicolumn{4}{c|}{indices} & $\text{norm}^2$ \\
        & $ia$ & $ib$ & $ja$ & $jb$ & $i$ & $j$ & $i$ & $j$ & $a$ & $b$ & \\
        \hline
        \hline
        \multirow{15}{*}{\small $v$}
        & \myR & \myR & \myR & \myR & \myC &  \myC &  \myRC  & \myRC  & \myR & \myR & $d^2$ \\
        & \myR & \myR & \myR & \myC &  \myC &  \myC &  \myRC  & \myRC  & \myR & \myRC  & $d$ \\
        & \myR & \myR & \myC &  \myR & \myC &  \myC &  \myRC  & \myRC  & \myRC  & \myR & $d$ \\
        & \myR & \myR & \myC &  \myC &  \myC &  \myC &  \myRC  & \myC &  \myRC  & \myRC  & $m$ \\
        & \myR & \myC &  \myR & \myR & \myC &  \myC &  \myRC  & \myRC  & \myR & \myRC  & $d$ \\
        & \myR & \myC &  \myR & \myC &  \myC &  \myC &  \myRC  & \myRC  & \myR & \myC &  $d^2$ \\
        & \myR & \myC &  \myC &  \myR & \myC &  \myC &  \myRC  & \myRC  & \myRC  & \myRC  & $1$ \\
        & \myR & \myC &  \myC &  \myC &  \myC &  \myC &  \myRC  & \myC &  \myRC  & \myC &  $m d$ \\
        & \myC &  \myR & \myR & \myR & \myC &  \myC &  \myRC  & \myRC  & \myRC  & \myR & $d$ \\
        & \myC &  \myR & \myR & \myC &  \myC &  \myC &  \myRC  & \myRC  & \myRC  & \myRC  & $1$ \\
        & \myC &  \myR & \myC &  \myR & \myC &  \myC &  \myRC  & \myRC  & \myC &  \myR & $d^2$ \\
        & \myC &  \myR & \myC &  \myC &  \myC &  \myC &  \myRC  & \myC &  \myC &  \myRC  & $m d$ \\
        & \myC &  \myC &  \myR & \myR & \myC &  \myC &  \myC &  \myRC  & \myRC  & \myRC  & $m$ \\
        & \myC &  \myC &  \myR & \myC &  \myC &  \myC &  \myC &  \myRC  & \myRC  & \myC &  $m d$ \\
        & \myC &  \myC &  \myC &  \myR & \myC &  \myC &  \myC &  \myRC  & \myC &  \myRC  & $m d$ \\
        \hline
    \end{tabular}
}
\end{minipage}
    \caption{Flattenings of $\AA_\phi$: $\sigma(\AA_\phi) = d\sqrt{m} \lor m$, $v(\AA_\phi) = \sqrt{md} \lor d$, used in~\eqref{eq:ellifitalpha}.}
    \label{tab:elli_fitt_alpha}
\end{table}

    \begin{enumerate}
        \item The decoupled version of $M_{\phi}$ is a gaussian matrix chaos of order $4$. Algorithm~\ref{alg:buildthetable} outputs Table~\ref{tab:elli_fitt_alpha}, and the iterated strong NCK inequality (Theorem~\ref{thm:iteratedfreenck}) and Theorem~\ref{thm:decoupling} yield
        \begin{equation}\label{eq:ellifitalpha}
            \E\norm{M_\phi} \lesssim \brap{d\sqrt{m} \lor m} + \log(md+m)^{3} \brap{\sqrt{md} \lor d}.
        \end{equation}
        
        \item The decoupled version of $M_{\psi}$ is a matrix chaos of order $2$ whose random variables are given by $h_{i, a} = g_{i, a}^2 - 1$. Algorithm~\ref{alg:buildthetable} outputs Table~\ref{tab:elli_fitt_beta}, and the iterated strong matrix Rosenthal inequality (Theorem~\ref{thm:iteratedstrongrose}) and Theorem~\ref{thm:decoupling} yield
        \begin{equation}\label{eq:ellifitbeta}
            \E\norm{M_\psi} \lesssim \brap{m \lor \sqrt{md}} + \log(md+m)^{\frac{9}{2}} \brap{\sqrt{m} \lor \sqrt{d}}.
        \end{equation}
        
\begin{table}[t]
    \centering
    \begin{tabular}{| c | c c : c c | c c c | c |}
        \hline
        & \multicolumn{4}{c|}{coordinates} & \multicolumn{3}{c|}{summation} & \\
        type & \multicolumn{2}{c:}{chaos} & \multicolumn{2}{c|}{matrix} & \multicolumn{3}{c|}{indices} & $\text{norm}^2$ \\
        & $ia$ & $ja$ & $i$ & $j$ & $i$ & $j$ & $a$ & \\
        \hline
        \hline
        \multirow{4}{*}{$\sigma$}
        & \myceco \myR & \myceco \myR & \myceco \myR & \myceco \myC &  \myceco \myR & \myceco \myRC  & \myceco \myR & \myceco $m d$ \\
        & \myceco \myR & \myceco \myC &  \myceco \myR & \myceco \myC &  \myceco \myR & \myceco \myC &  \myceco \myRC  & \myceco $m^2$ \\
        & \myC &  \myR & \myR & \myC &  \myRC  & \myRC  & \myRC  & $1$ \\
        & \myceco \myC &  \myceco \myC &  \myceco \myR & \myceco \myC &  \myceco \myRC  & \myceco \myC &  \myceco \myC &  \myceco $m d$ \\
        \hline
        \multirow{3}{*}{$v$}
        & \myR & \myR & \myC &  \myC &  \myRC  & \myRC  & \myR & $d$ \\
        & \myR & \myC &  \myC &  \myC &  \myRC  & \myC &  \myRC  & $m$ \\
        & \myC &  \myR & \myC &  \myC &  \myC &  \myRC  & \myRC  & $m$ \\
        \hline
    \end{tabular}
    \caption{Flattenings of $\AA_\psi$: $\sigma(\AA_\psi) = m \lor \sqrt{md}$, $v(\AA_\psi) = \sqrt{m} \lor \sqrt{d}$, used in~\eqref{eq:ellifitbeta}.} 
    \label{tab:elli_fitt_beta}
\end{table}

    \end{enumerate}
    The first term in \eqref{eq:ellifitalpha} and in \eqref{eq:ellifitbeta} dominates under the  assumption on $d,m$, concluding the proof.
\end{proof}

\subsection*{Acknowledgements}
ASB would like to thank Sam Hopkins, Ankur Moitra, and Holger Rauhut for asking insightful questions in conversations over the past few years that helped motivate the methods of this paper. RvH was supported in part by NSF grant DMS-2347954.

\bibliographystyle{alpha}
\bibliography{ref}

\addtocontents{toc}{\protect\setcounter{tocdepth}{1}}%
\setcounter{tocdepth}{1}%

\appendix
%%%%%%%%%%%%%%%%%%%%%%%%%%%%%%%%%%%%%%%%
\section{Proofs of main results and supporting lemmas}\label{section:proofs}
%%%%%%%%%%%%%%%%%%%%%%%%%%%%%%%%%%%%%%%%

%%%%%%%%%%%%%%%%%%%%%%%%%%%%%%%%%%%%%%%%
\subsection{The iteration scheme}
%%%%%%%%%%%%%%%%%%%%%%%%%%%%%%%%%%%%%%%%

The basic approach to all our main results was outlined in Section~\ref{subsec:iterapproach}. For each of the iterated inequalities, we start with an inequality for linear random matrices (i.e., for chaos of order $q=1$). The linear inequalities involve four parameters $\sigma_R,\sigma_C,v,r$ defined in Section~\ref{subsubsec:flattlincase}. Applying these bounds conditionally on all but one of the chaos coordinates gives rise to four intermediate flattenings as shown in Figure~\ref{fig:branching}. As the intermediate flattenings are themselves matrix chaoses of smaller order, the proofs proceed by induction.

The following lemma formalizes the fact, used in the induction step, that the final flattenings of intermediate flattenings coincide with the final flattenings of the original chaos. 

\begin{lemma}[$\sigma$, $v$ and $r$ of intermediate flattenings]\label{lem:finalmixedflattenings}
Let $Y$ be a decoupled chaos as in~\eqref{eq:decnongaussianchaos}. Given an intermediate flattening $\interflatty{Z}{R}{C}$, which is a chaos of order $\abs{Z}$, we have
    \begin{equation*}\label{eq:sigmaintermediatedefinition}
        \sigma\brap{\interflatta{Z}{R}{C}} = \max_{R' \sqcup C' = Z} \norm{\flatta{R \cup R'}{C \cup C'}},
    \end{equation*}
    \begin{equation*}\label{eq:vintermediatedefinition}
        v\brap{\interflatta{Z}{R}{C}} = \max_{\substack{R' \sqcup C' = Z\\R'\neq\emptyset}} \norm{\flatta{R'}{R \cup C \cup C'}},
    \end{equation*}
    \begin{equation*}\label{eq:rintermediatedefinition}
        r\brap{\interflatta{Z}{R}{C}} = \max_{\substack{R' \cup C' = Z\\R' \cap C' \neq \emptyset}} \norm{\flatta{R \cup R'}{C \cup C'}}.
    \end{equation*}
\end{lemma}

\begin{proof}
By its definition \eqref{eq:intermediateflattenings}, the intermediate flattening
$\interflatty{Z}{R}{C}$ is a matrix chaos with the same coefficients $\mathcal{A}_{i_1,\ldots,i_{q+2}}$ as $Y$, but where the $|Z|$ chaos coordinates are indexed by $i_t$ for $t\in Z$ and the two matrix coordinates are indexed by $(i_t:t\in R)$ and $(i_t:t\in C)$, respectively. The conclusion now follows readily from the definitions 
\eqref{eq:sigmadefinition}, \eqref{eq:vdefinition} and \eqref{eq:rdefinition} of the chaos parameters.
\end{proof}

For completeness, we record here two basic relations between the chaos parameters.

\begin{lemma}\label{lem:comparison_v_r}
For any chaos $Y$ as in~\eqref{eq:decnongaussianchaos}, we have
    \begin{equation*}
        r(\AA) \leq \sigma(\AA) \quad \text{ and }\quad r(\AA) \leq v(\AA).
    \end{equation*}
\end{lemma}

\begin{proof}
In the case $q=1$, we may readily read off from the expressions in 
section \ref{subsubsec:flattlincase} that
$$
     r(Y) = \norm{\flatta{\onewithcolor, 2}{\onewithcolor, 3}}
     \le \norm{\flatta{\onewithcolor, 2}{3}} \le \sigma(Y),
\qquad
     r(Y) = \norm{\flatta{\onewithcolor, 2}{\onewithcolor, 3}} \le 
     \norm{\flatta{\onewithcolor}{2, 3}} = v(Y),
$$
where we used that $\|A_i\| \le \|A_i\|_F=\|{\vec(A_i)}\|$ in the second inequality.

Now let $q\ge 2$, and let $\flatta{R}{C}$ be any $r$-flattening.
Consider a $3$-tensor $\BB$ whose entries are given by
$\BB_{\vect{i}_{R \cap C}, \vect{i}_{R \setminus C}, \vect{i}_{C \setminus R}} = 
\AA_{\vect{i}}$,
where $\vect{i} = \brap{i_1, \ldots, i_{q+2}}$ and $\vect{i}_T = \brap{i_t \colon t \in T}$.
    Then
    \begin{equation*}
        \flattb{\onewithcolor, 2}{\onewithcolor, 3} = \flatta{R}{C},\quad
        \flattb{\onewithcolor, 2}{3} = \flatta{R}{C \setminus R},\quad
        \flattb{\onewithcolor}{2, 3} = \flatta{R \cap C}{R^c \cup C^c}.
    \end{equation*}
    As $\flatta{R}{C \setminus R}$ is a $\sigma$-flattening of $\AA$ and $\flatta{R \cap C}{R^c \cup C^c}$ is a $v$-flattening of $\AA$,
    we obtain
    \begin{equation*}
        \norm{\flatta{R}{C}} = \norm{\flattb{\onewithcolor, 2}{\onewithcolor, 3}} \leq \norm{\flattb{\onewithcolor, 2}{3}} = \norm{\flatta{R}{C \setminus R}} \leq \sigma(\AA)
    \end{equation*}
    \begin{equation*}
        \norm{\flatta{R}{C}} = \norm{\flattb{\onewithcolor, 2}{\onewithcolor, 3}} \leq \norm{\flattb{\onewithcolor}{2, 3}} = \norm{\flatta{R \cap C}{R^c \cup C^c}} \leq v(\AA)
    \end{equation*}
    by applying the inequalities for the case $q=1$ to the tensor $\BB$.
\end{proof}

%%%%%%%%%%%%%%%%%%%%%%%%%%%%%%%%%%%%%%%%
\subsection{Proof of Iterated NCK}\label{sec:proofIterNCK}
%%%%%%%%%%%%%%%%%%%%%%%%%%%%%%%%%%%%%%%%

We start by stating the linear theorem. The following result is classical, but we spell it out in a slightly more general setting than is customary.

\begin{theorem}[Noncommutative Khintchine (NCK) inequality]\label{thm:nck}
    Let $X = \sum_{i\in[m]} h_{i} A_{i}$, where $h_1,\dots,h_m$ are i.i.d.\ copies of a centered random variable $h$ and $A_1,\ldots,A_m$ are $d_1\times d_2$ matrix coefficients (we define $d \coloneqq d_1 \lor d_2$). Then we have
$$
    \|h\|_{L^1}\brap{\sigma_R(X)+\sigma_C(X)}\lesssim \E\norm{X} \lesssim 
    \|h\|_{\psi_2}\log(d)^{\frac{1}{2}} \brap{\sigma_R(X)+\sigma_C(X)}.
$$
Alternatively, the upper bound remains valid if $\|h\|_{\psi_2}$ is replaced by
$\|h\|_{L^{\log m}}$.
\end{theorem}

\begin{proof}
We begin with the lower bound.
Let $\varepsilon_i$ be i.i.d.\ Rademachers variables independent of $h_i$, and define
$\tilde X = \sum_{i\in[m]}\varepsilon_i h_{i} A_{i}$. Then $\E\|\tilde X\| \le 2\,\E\norm{X}$
by a standard symmetrization argument \cite[Lemma 6.3.2]{vershynin2018book}.  
Taking the expectation only with respect to $\boldsymbol{\varepsilon}$, we can estimate
$$
    \E_{\boldsymbol{\varepsilon}} \|\tilde X\|
    \gtrsim 
    2\brap{\E_{\boldsymbol{\varepsilon}} \|\tilde X\|^2}^{\frac{1}{2}}
    =
    \brap{\E_{\boldsymbol{\varepsilon}} \|\tilde X^\top\tilde X\|}^{\frac{1}{2}}+
    \brap{\E_{\boldsymbol{\varepsilon}} \|\tilde X\tilde X^\top\|}^{\frac{1}{2}}
    \geq
    {\textstyle\norm{\sum_{i} h_i^2 A_i^\top A_i}^{\frac{1}{2}} +
    \norm{\sum_{i} h_i^2 A_iA_i^\top}^{\frac{1}{2}}},
$$
where we used the Khintchine-Kahane inequality \cite{latala1994kahane} in the first step and Jensen's inequality in the last step. As the right-hand side is a convex function of $(|h_1|,\ldots,|h_m|)$, taking the expecation and applying Jensen's inequality
yields $\E\|X\| \gtrsim \|h\|_{L^1}\brap{\sigma_R(X)+\sigma_C(X)}$.

We now turn to the upper bound. The classical form of the noncommutative Khintchine inequality \cite[\S 9.8]{Pisier2003IntroductionTO} (and $\Tr[|M|^p]^{1/p}\lesssim\|M\|$ for any $d\times d$ matrix $M$ and $p\gtrsim\log d$) yields the upper bound in the special case that $h_i$ are Rademacher or standard Gaussian variables.
The subgaussian upper bound then follows as $\E\|X\|\lesssim \|h\|_{\psi_2}\E\|X_G\|$, where
$X_G=\sum_{i\in[m]}g_{i} A_{i}$ with $g_1,\ldots,g_m$ i.i.d.\ standard Gaussians,
by the subgaussian comparison theorem \cite[Corollary 8.6.3]{vershynin2018book}.

Alternatively, by symmetrizing as in the lower bound, we can estimate
\begin{align*}
    \E\norm{X} \le 2\,\E\|\tilde X\| &\lesssim
    \log(d)^{\frac{1}{2}} \brap{\textstyle
    \E\norm{\sum_{i} h_i^2 A_i^\top A_i}^{\frac{1}{2}} +
    \E\norm{\sum_{i} h_i^2 A_iA_i^\top}^{\frac{1}{2}}} \\
    & \le \E\brap{\max_{i\in[m]}|h_i|} \log(d)^{\frac{1}{2}} \brap{\sigma_R(X)+\sigma_C(X)}
\end{align*}
by applying the Rademacher form of NCK conditionally on $\boldsymbol{h}$.
It remains to note that we can estimate $\E\max_{i\in[m]}|h_i| \le (\E\max_{i\in[m]}|h_i|^p)^{1/p} \le (m\E|h|^p)^{1/p} \lesssim \|h\|_{L^p}$ for $p=\log m$.
\end{proof}

We can now complete the proof of Theorem~\ref{thm:iteratednck}.

\begin{proof}[\textbf{Proof of Theorem~\ref{thm:iteratednck}}]
    The proof is by induction on $q$. The base case $q = 1$ is given by Theorem~\ref{thm:nck}. For the induction step, let $q \geq 2$. We start with the lower bound.
    
    If we condition on $\vect{h}^{(1)}, \ldots, \vect{h}^{(q-1)}$, and treat $Y$ as a linear chaos, then applying the lower bound in NCK with respect to the randomness of $\vect{h}^{(\qwithcolor)}$ yields (see Figure~\ref{fig:branching})
    \begin{equation*}
        \E_{\vect{h}^{(\qwithcolor)}}\norm{Y} =
        \E_{\vect{h}^{(\qwithcolor)}}\norm{\interflatty{1:q}{q+1}{q+2}} 
        \gtrsim \|h\|_{L^1}\brap{
        \norm{\interflatty{1:q-1}{\qwithcolor,q+1}{q+2}} + \norm{\interflatty{1:q-1}{q+1}{\qwithcolor, q+2}}}.
    \end{equation*}
    Taking expectations and using the induction hypothesis yields
    \begin{equation*}
        \E\norm{Y} \qgtrsim \|h\|_{L^1}^q\brap{
        \sigma\brap{\interflatta{1:q-1}{\qwithcolor,q+1}{q+2}} \lor \sigma\brap{\interflatta{1:q-1}{q+1}{\qwithcolor, q+2}}},
    \end{equation*}
    and the conclusion follows from Lemma \ref{lem:finalmixedflattenings}.

    The proof of the upper bound follows similarly. The NCK upper bound yields
    \begin{equation*}
        \E_{\vect{h}^{(\qwithcolor)}}\norm{Y} 
        \lesssim \|h\|_{\psi_2}\log(d)^{\frac{1}{2}}\brap{ \norm{\interflatty{1:q-1}{\qwithcolor,q+1}{q+2}} + \norm{\interflatty{1:q-1}{q+1}{\qwithcolor, q+2}}}.
    \end{equation*}
    Taking expectations and using the induction hypothesis, we obtain
    \begin{equation*}
        \E\norm{Y} 
        \qlesssim \|h\|_{\psi_2}^q\log(d)^{\frac{1}{2}}\log(dm+m)^{\frac{q-1}{2}} \brap{
        \sigma\brap{\interflatta{1:q-1}{\qwithcolor,q+1}{q+2}} \lor \sigma\brap{\interflatta{1:q-1}{q+1}{\qwithcolor, q+2}} },
    \end{equation*}
    where we used that the largest dimension of the intermediate flattenings is at most $dm$. The conclusion follows from Lemma \ref{lem:finalmixedflattenings} and using $\log(d)^{\frac{1}{2}}\log(dm+m)^{\frac{q-1}{2}}\qlesssim \log(d+m)^{\frac{q}{2}}$.   
    The identical proof yields the variant of the upper bound where $\|h\|_{\psi_2}$ is replaced by $\|h\|_{L^{\log m}}$.
\end{proof}

\begin{remark}
\label{rem:nckwithc}
It is readily verified in the proof that the iterated NCK inequality also remains
valid if $\|h\|_{\psi_2}$ is replaced by $C_c\|h\|_{L^{c\log m}}$ for any $c>0$, where $C_c$ is a constant that depends on $c$ only. This variant will be used below in the proofs of the iterated Rosenthal inequalities.
\end{remark}

%%%%%%%%%%%%%%%%%%%%%%%%%%%%%%%%%%%%%%%%
\subsection{Proof of iterated strong NCK}\label{sec:proofIterStrongNCK}
%%%%%%%%%%%%%%%%%%%%%%%%%%%%%%%%%%%%%%%%

We start by stating the linear theorem.

\begin{theorem}[Strong Noncommutative Khintchine inequality]\label{thm:freenck}
    Let $X = \sum_{i\in[m]} h_{i} A_{i}$, where $h_1,\dots,h_m$ are i.i.d.\ copies of a centered random variable $h$ and $A_1,\ldots,A_m$ are $d_1\times d_2$ matrix coefficients (we define $d \coloneqq d_1 \lor d_2$). Then we have
    \begin{equation*}
        \E\norm{X} \lesssim \|h\|_{\psi_2}\brap{\sigma_R(X) + \sigma_C(X) + \log(d)^{\frac{3}{2}} v(X)}.
    \end{equation*}
\end{theorem}

\begin{proof}
We may estimate $\E\norm{X}\lesssim \|h\|_{\psi_2}\E\norm{X_G}$ as in the proof of Theorem \ref{thm:nck}. For the Gaussian random matrix $X_G$, applying \cite[Corollary 2.2 and Lemma 2.5]{bandeira2023matrix} yields
$$
    \E\norm{X_G} \lesssim \sigma_R(X) + \sigma_C(X) + \log(d)^{\frac{3}{4}}(\sigma_R(X)\vee\sigma_C(X))^{\frac{1}{2}}v(X)^{\frac{1}{2}},
$$
and the conclusion follows by applying Young's inequality to the last term.
\end{proof}

We can now complete the proof of Theorem~\ref{thm:iteratedfreenck}.

\begin{proof}[\textbf{Proof of Theorem~\ref{thm:iteratedfreenck}}]
    The proof is by induction on $q$. The base case $q = 1$ is given by Theorem~\ref{thm:freenck}. For the induction step, let $q \geq 2$. 
    If we condition on $\vect{h}^{(1)}, \ldots, \vect{h}^{(q-1)}$, and treat $Y$ as a linear chaos, then applying Theorem \ref{thm:freenck} with
    respect to $\vect{h}^{(\qwithcolor)}$ yields (see Figure~\ref{fig:branching})
    \begin{equation*}
        \E_{\vect{h}^{(\qwithcolor)}}\norm{Y} \lesssim 
        \|h\|_{\psi_2}\brap{\norm{\interflatty{1:q-1}{\qwithcolor,q+1}{q+2}} + \norm{\interflatty{1:q-1}{q+1}{\qwithcolor, q+2}} + \log(d)^{\frac{3}{2}} \norm{\interflatty{1:q-1}{\qwithcolor}{q+1, q+2}}}.
    \end{equation*}
    We now take the expectation and bound the norm of each intermediate flattening.
    For the first two terms, we use the induction hypothesis to estimate
    \begin{align*}
    &\E \norm{\interflatty{1:q-1}{\qwithcolor,q+1}{q+2}} \qlesssim
    \|h\|_{\psi_2}^{q-1}\brap{
    \sigma\brap{\interflatta{1:q-1}{\qwithcolor,q+1}{q+2}} +
        \log(dm+m)^{\frac{q+1}{2}} v\brap{\interflatta{1:q-1}{\qwithcolor,q+1}{q+2}}},
        \\
    &\E \norm{\interflatty{1:q-1}{q+1}{\qwithcolor,q+2}} \qlesssim
    \|h\|_{\psi_2}^{q-1}\brap{
    \sigma\brap{\interflatta{1:q-1}{q+1}{\qwithcolor,q+2}} +
        \log(dm+m)^{\frac{q+1}{2}} v\brap{\interflatta{1:q-1}{q+1}{\qwithcolor,q+2}}}.
    \end{align*}
    For the last term, we use the iterated NCK inequality (Theorem~\ref{thm:iteratednck}) to estimate
    \begin{equation*}
        \E \norm{\interflatty{1:q-1}{\qwithcolor}{q+1, q+2}} \qlesssim
        \|h\|_{\psi_2}^{q-1}
        \log(d^2\vee m+m)^{\frac{q-1}{2}}
        \sigma\brap{\interflatta{1:q-1}{\qwithcolor}{q+1, q+2}}.
    \end{equation*} 
    Here we used that the largest dimension of the first two intermediate flattenings is at most $dm$ and of the last intermediate flattening is at most $d^2\vee m$.
    To conclude, it remains to apply Lemma~\ref{lem:finalmixedflattenings}, and to  note that all flattenings that appear in the leading terms are of $\sigma$-type and that all flattenings that appear in the terms with logarithmic factors are of $v$-type.
\end{proof}

%%%%%%%%%%%%%%%%%%%%%%%%%%%%%%%%%%%%%%%%
\subsection{Proof of iterated Rosenthal inequality}\label{sec:proofIteratedRosenthal}
%%%%%%%%%%%%%%%%%%%%%%%%%%%%%%%%%%%%%%%%

We begin by stating the linear theorem. The upper bound follows from the matrix Rosenthal inequality that may be found in \cite{junge2013,Mackey_2014} (see also \cite[Example~2.15]{Brailovskaya2022UniversalityAS}). We were unable to locate a reference
for the lower bound.

\begin{theorem}[Matrix Rosenthal inequality]\label{thm:rosenthal}
    Let $X = \sum_{i\in[m]} h_{i} A_{i}$ where $h_1,\dots,h_m$ are i.i.d.\ copies of a centered unit-variance random variable $h$, and $A_i$ are $d_1\times d_2$ matrix coefficients (set $d \coloneqq d_1 \lor d_2$). Let $\alpha_c(h) \coloneqq \|h\|_{L^{c\log(d+m)}}$ for $c>0$. Then we have
    \begin{equation*}
        \E\norm{X} \lesssim_c \log(d+m)^{\frac{1}{2}} \brap{\sigma_R(X)+\sigma_C(X)} + \alpha_c(h) \log(d+m) \,r(X)
    \end{equation*}
and
    \begin{equation*}
        \E\norm{X} \gtrsim
    \sigma_R(X)+\sigma_C(X) - C_c \alpha_c(h) \log(d+m)^{\frac{1}{2}} \,r(X),
    \end{equation*}
    where $C_c$ is a constant that depends only on $c$.
\end{theorem}

\begin{proof}
The upper bound follows by applying \cite[Example~2.15 and Remark~2.1]{Brailovskaya2022UniversalityAS} with $2p=\lfloor c\log(d+m)\rfloor$ (and $\Tr[|M|^p]^{1/p}\lesssim\|M\|$ for any $d\times d$ matrix $M$ and $p\gtrsim\log d$).

For the lower bound, we begin by estimating 
\begin{align*}
    \E\|X\|
    &\gtrsim
    {\textstyle \E\norm{\sum_{i} h_i^2 A_i^\top A_i}^{\frac{1}{2}} +
    \E\norm{\sum_{i} h_i^2 A_iA_i^\top}^{\frac{1}{2}}}
     \\ &   \geq
        \sigma_R(X)+\sigma_C(X) -
        {\textstyle\E\norm{\sum_{i} (h_i^2-1) A_i^\top A_i}^{\frac{1}{2}} -
    \E\norm{\sum_{i} (h_i^2-1) A_iA_i^\top}^{\frac{1}{2}}},
\end{align*}
where the first line follows from the proof of Theorem \ref{thm:nck} and the second line uses the triangle inequality. We can now apply the matrix Rosenthal upper bound to estimate
$$
    {\textstyle\E\norm{\sum_{i} (h_i^2-1) A_i^\top A_i}
        \lesssim_c
        \log(d+m)^{\frac{1}{2}} \sigma_R(X)r(X)
        + \alpha_c(h)^2 \log(d+m) r(X)^2,
        }
$$
where we used $\|\sum_{i} (A_i^\top A_i)^2\|\le \sigma_R(X)^2r(X)^2$. 
Estimating the remaining term similarly, we obtain
$$
    \E\|X\| \gtrsim \sigma_R(X)+\sigma_C(X) - C_c
    \log(d+m)^{\frac{1}{4}} (\sigma_R(X)+\sigma_C(X))^{\frac{1}{2}}r(X)^{\frac{1}{2}}
    - C_c    \alpha_c(h) \log(d+m)^{\frac{1}{2}} r(X)
$$
and the conclusion follows by Young's inequality.
\end{proof}

We can now complete the proof of Theorem~\ref{thm:iteratedrose}.

\begin{proof}[\textbf{Proof of Theorem~\ref{thm:iteratedrose}}]
We first prove the upper bound. We aim to prove by induction on $q$ that
$$
        \E\norm{Y} \lesssim_{q,c} \log(d+m)^{\frac{q}{2}} \sigma(\AA)+ \alpha_c(h)^q \log(d+m)^{\frac{q+1}{2}}r(\AA),
$$
for every $c>0$, from which the conclusion follows by choosing $c=1$.

The base case $q = 1$ is given by Theorem~\ref{thm:rosenthal}. For the induction step, let $q \geq 2$. If we condition on $\vect{h}^{(1)}, \ldots, \vect{h}^{(q-1)}$, then applying Theorem \ref{thm:rosenthal} with
    respect to $\vect{h}^{(\qwithcolor)}$ yields (see Figure~\ref{fig:branching})
    \begin{align*}
        &\E_{\vect{h}^{(\qwithcolor)}}\norm{Y} \lesssim_c  \log(d + m)^{\frac{1}{2}}\brap{\norm{\interflatty{1:q-1}{\qwithcolor,q+1}{q+2}} +  \norm{\interflatty{1:q-1}{q+1}{\qwithcolor, q+2}}} \\
        &\qquad\qquad\qquad + \alpha_c(h) \log(d + m) \norm{\interflatty{1:q-1}{\qwithcolor, q+1}{\qwithcolor, q+2}}.
    \end{align*}
Now note that the largest dimension of the intermediate flattenings that appear above
is $md$. As $\frac{c}{2}\log(md+m)\le c\log(d+m)$, the induction hypothesis with $c\leftarrow \frac{c}{2}$ and Lemma~\ref{lem:finalmixedflattenings} yield
\begin{align*}
\E\norm{\interflatty{1:q-1}{\qwithcolor,q+1}{q+2}} &\lesssim_{q,c}
\log(md+m)^{\frac{q-1}{2}} \sigma\brap{\AA} +
\alpha_c(h)^{q-1}\log(md+m)^{\frac{q}{2}} r\brap{\AA}, \\
\E\norm{\interflatty{1:q-1}{q+1}{\qwithcolor, q+2}} &\lesssim_{q,c}
\log(md+m)^{\frac{q-1}{2}} \sigma\brap{\AA} +
\alpha_c(h)^{q-1}\log(md+m)^{\frac{q}{2}} r\brap{\AA}.
\end{align*}
On the other hand, the iterated NCK inequality (Theorem \ref{thm:iteratednck} and
Remark \ref{rem:nckwithc}) yields
$$
    \EE\norm{\interflatty{1:q-1}{\qwithcolor, q+1}{\qwithcolor, q+2}}
    \lesssim_{q,c} \alpha_c(h)^{q-1}\log(md+m)^{\frac{q-1}{2}}
    \sigma\brap{\interflatta{1:q-1}{\qwithcolor, q+1}{\qwithcolor, q+2}}.
$$
Using again Lemma~\ref{lem:finalmixedflattenings} yields $\sigma(\interflatta{1:q-1}{\qwithcolor, q+1}{\qwithcolor, q+2})\le r(\AA)$. The proof of the upper bound is readily concluded by combining the above estimates.

The proof of the lower bound is very similar. We first estimate
$$
        \E_{\vect{h}^{(\qwithcolor)}}\norm{Y} \gtrsim  \norm{\interflatty{1:q-1}{\qwithcolor,q+1}{q+2}} +  \norm{\interflatty{1:q-1}{q+1}{\qwithcolor, q+2}}
         - C_c\alpha_c(h) \log(d + m)^{\frac{1}{2}} \norm{\interflatty{1:q-1}{\qwithcolor, q+1}{\qwithcolor, q+2}}
$$
using Theorem \ref{thm:rosenthal}. The proof is concluded by lower bounding the expectation of the first two terms using the induction hypothesis and Lemma \ref{lem:finalmixedflattenings}, and bounding the expectation of the last term by the iterated NCK inequality as in the upper bound.
\end{proof}

%%%%%%%%%%%%%%%%%%%%%%%%%%%%%%%%%%%%%%%%
\subsection{Proof of iterated strong matrix Rosenthal inequality}\label{sec:proof:thm:iteratedstrongrose}
%%%%%%%%%%%%%%%%%%%%%%%%%%%%%%%%%%%%%%%%

We first state the linear theorem.

\begin{theorem}[Strong matrix Rosenthal inequality]\label{thm:linearuniversalityrosenthal}
    Let $X = \sum_{i\in[m]} h_{i} A_{i}$ where $h_1,\dots,h_m$ are i.i.d.\ copies of a centered unit-variance random variable $h$, and $A_i$ are $d_1\times d_2$ matrix coefficients (set $d \coloneqq d_1 \lor d_2$). Let $\alpha_c(h) \coloneqq \|h\|_{L^{c\log(d+m)}}$ for $c>0$. Then we have
    \begin{equation*}
        \E \norm{X} \lesssim_c \sigma_R(X)+\sigma_C(X) + \alpha_c(h) \log(d+m)^2 v(X).
    \end{equation*}
\end{theorem}

\begin{proof}
Combining
\cite[Theorem~2.9~and~Remark~2.1]{Brailovskaya2022UniversalityAS} and
\cite[Theorem~2.7~and~Lemma~2.5]{bandeira2023matrix} with $2p=\lfloor c\log(d+m)\rfloor$ (and $\Tr[|M|^p]^{1/p}\lesssim\|M\|$ for any $d\times d$ matrix $M$ and $p\gtrsim\log d$)
yields
$$
\E\norm{X} \lesssim_c \sigma_R(X)+\sigma_C(X)+\log(d+m)^{\frac{3}{4}}(\sigma_R(X)\vee\sigma_C(X))^{\frac{1}{2}}v(X)^{\frac{1}{2}} +
\alpha_c(h)\log(d+m)^2r(X).
$$
The conclusion follows using Young's inequality and as $r(X)\le v(X)$ (Lemma~\ref{lem:comparison_v_r}).
\end{proof}

We can now complete the proof of Theorem~\ref{thm:iteratedstrongrose}.

\begin{proof}[\textbf{Proof of Theorem~\ref{thm:iteratedstrongrose}}]
We aim to prove by induction on $q$ that
$$
        \E\norm{Y} \lesssim_{q,c} \sigma(\AA)+ \alpha_c(h)^q \log(d+m)^{\frac{q+3}{2}}v(\AA),
$$
for every $c>0$, from which the conclusion follows by choosing $c=1$.

The base case $q = 1$ is given by Theorem~\ref{thm:linearuniversalityrosenthal}. For the induction step, let $q \geq 2$. If we condition on $\vect{h}^{(1)}, \ldots, \vect{h}^{(q-1)}$, then applying Theorem \ref{thm:linearuniversalityrosenthal} with
    respect to $\vect{h}^{(\qwithcolor)}$ yields (see Figure~\ref{fig:branching})
$$
        \E_{\vect{h}^{(\qwithcolor)}}\norm{Y} \lesssim_c  \norm{\interflatty{1:q-1}{\qwithcolor,q+1}{q+2}} + \norm{\interflatty{1:q-1}{q+1}{\qwithcolor, q+2}} + \alpha_c(h)\log(d + m)^2 \norm{\interflatty{1:q-1}{\qwithcolor}{q+1, q+2}}.
$$
As in the proof of Theorem \ref{thm:iteratedrose}, the induction hypothesis with $c\leftarrow \frac{c}{2}$ and Lemma~\ref{lem:finalmixedflattenings} yield
\begin{align*}
    \E\norm{\interflatty{1:q-1}{\qwithcolor,q+1}{q+2}} &\lesssim_{q,c}
    \sigma(\AA)+ \alpha_c(h)^{q-1} \log(md+m)^{\frac{q+2}{2}}v(\AA),\\
    \E\norm{\interflatty{1:q-1}{q+1}{\qwithcolor, q+2}} &\lesssim_{q,c}
    \sigma(\AA)+ \alpha_c(h)^{q-1} \log(md+m)^{\frac{q+2}{2}}v(\AA).
\end{align*}
On the other hand, the iterated NCK inequality (Theorem \ref{thm:iteratednck} and
Remark \ref{rem:nckwithc}) yields
$$
    \EE\norm{\interflatty{1:q-1}{\qwithcolor}{q+1, q+2}}
    \lesssim_{q,c} \alpha_c(h)^{q-1}\log(md+m)^{\frac{q-1}{2}}
    \sigma\brap{\interflatta{1:q-1}{\qwithcolor}{q+1, q+2}}.
$$
Using again Lemma~\ref{lem:finalmixedflattenings} yields $\sigma(\interflatty{1:q-1}{\qwithcolor}{q+1, q+2})\le v(\AA)$, concluding the proof.
\end{proof}

\begin{remark}
In the strong matrix Rosenthal inequality that appears in the proof of Theorem \ref{thm:linearuniversalityrosenthal}, the distributional parameter $\alpha_c(h)$ appears
only in the term that is controlled by $r(X)$. We simplified this inequality by estimating
$r(X)\le v(X)$. This simplification can be lossy when $r(X)\ll v(X)$, particularly in sparse situations when the parameter $\alpha_c(h)$ may be very large.

One may hope that exploiting the sharper form of the strong matrix Rosenthal inequality   could give rise to an improved form of Theorem \ref{thm:iteratedstrongrose} where the distributional parameter appears only in a term controlled by $r(\AA)$. It is not possible to iterate the sharper inequality, however,
as it is not true in general that $r(\interflatta{1:q-1}{\qwithcolor}{q+1, q+2})$ can be controlled by $r(\AA)$. 

It is possible to obtain improved chaos inequalities by introducing additional chaos parameters that control such terms, but we do not at present know of a compelling application of such inequalities.
\end{remark}

%%%%%%%%%%%%%%%%%%%%%%%%%%%%%%%%%%%%%%%%
\subsection{Norms of flattenings}\label{sec:proof:prop:flattenings_combinatorial}
%%%%%%%%%%%%%%%%%%%%%%%%%%%%%%%%%%%%%%%%

We first consider chaoses of combinatorial type.

\begin{proof}[\textbf{Proof of Proposition~\ref{prop:flattenings_combinatorial}}]
    Given a chaos of combinatorial type~\eqref{eq:chaoscombinatorialtype}, its  flattenings~\eqref{eq:flattenings} are 
    \begin{equation*}
        \flatta{R}{C} = \sum_{\s\in[S_1]\times\cdots\times[S_p]}  \brap{\bigotimes_{t \in R} e_{I_t(\s)}} \otimes \brap{\bigotimes_{t \in C} e_{I_t(\s)}^\adj}.
    \end{equation*}
    Using the natural identification $e_{I_t(\s)} \simeq e_{I_t(\s)_1} \otimes \cdots \otimes e_{I_t(\s)_{\abs{I_t}}}$ and permuting the order of tensor products (which corresponds to reordering rows and columns), we obtain
    \begin{equation}\label{eq:flattenings_expressed}
        \flatta{R}{C} \simeq \sum_{\s\in[S_1]\times\cdots\times[S_p]} \bigotimes_{u \in [p]} \brap{e_{s_u}^{\otimes \mu_u} \otimes (e_{s_u}^\adj)^{\otimes \nu_u}},
    \end{equation}
    where $\mu_u$ and $\nu_u$ denote the number of times the summation index $s_u$ appears as a row or column index respectively in the tensor product. Distributivity yields
    \begin{equation*}
        \flatta{R}{C} \simeq \bigotimes_{u \in [p]} B_u\qquad\text{with}\qquad
        B_u =  \sum_{s \in [S_u]} \brap{e_{s}^{\otimes \mu_u} \otimes (e_{s}^\adj)^{\otimes \nu_u}}.
    \end{equation*}
In particular, we have
    \begin{equation*}
        \norm{\flatta{R}{C}} = \prod_{u \in [p]} \norm{B_u}.
    \end{equation*}
Now note that, by definition, $\mu_u=0$ if and only if $u\in\RR^c$, and $\nu_u=0$ if and only if $u\in\CC^c$. We can therefore compute using \eqref{eq:mu_nu_cases} that
$\norm{B_u}=(\sqrt{S_u})^{1_{u\in\RR^c}+1_{u\in\CC^c}}$, and the conclusion follows.
\end{proof}

To proceed, we need the following.

\begin{lemma}\label{lem:comparison}
    If $M,M'$ are (real) matrices so that $\abs{M_{i,j}}\leq M_{i,j}'$ for all $i,j$,
    then $\norm{M}\le\norm{M'}$.
\end{lemma}

\begin{proof}
Note that $\|M\|=\sup_{\|x\|=\|y\|=1}\sum_{i,j} x_i M_{i,j} x_j\le
\sup_{\|x\|=\|y\|=1}\sum_{i,j} |x_i| |M_{i,j}| |x_j| \le \|M'\|$.
\end{proof}

We can now extend the bound to chaoses of nearly combinatorial type.

\begin{proof}[\textbf{Proof of Proposition~\ref{prop:flattenings_nearly_combinatorial}}]
    Given a chaos of nearly combinatorial type, following the same steps as in the proof of Proposition~\ref{prop:flattenings_combinatorial}, we obtain an analogue of~\eqref{eq:flattenings_expressed}:
    \begin{equation*}
        \flatta{R}{C}^f \simeq \sum_{\s} f(\s) \bigotimes_{u \in [p]} \brap{e_{s_u}^{\otimes \mu_u} \otimes (e_{s_u}^\adj)^{\otimes \nu_u}}.
    \end{equation*}
    Define the associated flattening of combinatorial type by replacing $f\leftarrow 1$:
    \begin{equation*}
        \flatta{R}{C} \simeq \sum_{\s} \bigotimes_{u \in [p]} \brap{e_{s_u}^{\otimes \mu_u} \otimes (e_{s_u}^\adj)^{\otimes \nu_u}} .
    \end{equation*}
    Lemma~\ref{lem:comparison} yields
    $\|\flatta{R}{C}^f\| \le \norm{f}_\infty \|\flatta{R}{C}\|$, and we conclude using  Proposition~\ref{prop:flattenings_combinatorial}.
\end{proof}

\begin{remark}[Analogue of Proposition~\ref{prop:flattenings_nearly_combinatorial} for intermediate flattenings]\label{rmk:interflatt_nearcombinatorial}
    Consider a chaos $Y$ of nearly combinatorial type, and let $\interflatty{Z}{R}{C}$ be an intermediate flattening as in~\eqref{eq:intermediateflattenings}. Then
    \begin{equation*}
        \interflatty{Z}{R}{C} \simeq \sum_{\s} f(\s) \brap{\prod_{t \in Z}h_{I_t(s)}^{(t)}} \bigotimes_{u \in [p]} \brap{e_{s_u}^{\otimes \mu_u} \otimes (e_{s_u}^\adj)^{\otimes \nu_u}}.
    \end{equation*}
    If $h$ is uniformly bounded, then we can argue as in the proof of  Proposition~\ref{prop:flattenings_nearly_combinatorial} that
    $$
        \norm{\interflatty{Z}{R}{C}} \le \norm{f}_\infty \norm{h}_\infty^{|Z|}\norm{\flatta{R}{C}}
        \le
        \norm{f}_\infty \norm{h}_\infty^{|Z|}
        \brap{\prod_{u \in \RR^c}\sqrt{S_u}} \brap{\prod_{u \in \CC^c}\sqrt{S_u}},
    $$
    where we used Proposition \ref{prop:flattenings_combinatorial} in the second inequality.

    Note that while the definitions of the parameters $\sigma(\AA),v(\AA),r(\AA)$ only involved flattenings $\flatta{R}{C}$ with $R\cup C=[q+2]$, this need not be the case
    when considering intermediate flattenings. This is not a problem, as neither the
    definitions nor the arguments in the proof rely on this assumption.
\end{remark}

%%%%%%%%%%%%%%%%%%%%%%%%%%%%%%%%%%%%%%%%
\subsection{Menger's theorem}
%%%%%%%%%%%%%%%%%%%%%%%%%%%%%%%%%%%%%%%%

The following classical result is used in Section \ref{sec:graphmatrices}.

\begin{theorem}[Menger's theorem, \cite{goring2002menger}]\label{thm:menger}
    Let $G$ be a finite graph and $U, V \subseteq V(G)$ be two subsets of vertices. We say $S$ is a $U$---\,$V$ vertex separator if all paths from $U$ to $V$ pass through $S$. Then the minimal size of a $U$---\,$V$ vertex separator equals the maximal number of vertex-disjoint paths from $U$ to $V$ that contain exactly one point in $U$ and one point in $V$.
\end{theorem}

It should be noted that $U,V$ need not be disjoint in Theorem \ref{thm:menger}. In this case, any vertex in $U\cap V$ defines a path from $U$ to $V$ of length one. Such a point/path must therefore be contained in any vertex separator, and in any maximal collection of disjoint paths as in the theorem statement.

\vspace*{1cm}

\end{document}